\documentclass[a4paper]{amsart}

\usepackage{amssymb,enumitem}
\usepackage[all]{xy}
\usepackage{hyperref,aliascnt}
\usepackage{mathtools,array}

\numberwithin{equation}{section}

\newtheorem{lma}{Lemma}[section]

\newaliascnt{thmCt}{lma}
\newtheorem{thm}[thmCt]{Theorem}
\aliascntresetthe{thmCt}

\newaliascnt{corCt}{lma}
\newtheorem{cor}[corCt]{Corollary}
\aliascntresetthe{corCt}

\newaliascnt{prpCt}{lma}
\newtheorem{prp}[prpCt]{Proposition}
\aliascntresetthe{prpCt}

\newtheorem*{thm*}{Theorem}
\newtheorem*{rmk*}{Remark}
\theoremstyle{definition}

\newaliascnt{pgrCt}{lma}
\newtheorem{pgr}[pgrCt]{}
\aliascntresetthe{pgrCt}

\newaliascnt{dfnCt}{lma}
\newtheorem{dfn}[dfnCt]{Definition}
\aliascntresetthe{dfnCt}

\newaliascnt{rmkCt}{lma}
\newtheorem{rmk}[rmkCt]{Remark}
\aliascntresetthe{rmkCt}

\newaliascnt{qstCt}{lma}
\newtheorem{qst}[qstCt]{Question}
\aliascntresetthe{qstCt}

\newaliascnt{exaCt}{lma}
\newtheorem{exa}[exaCt]{Example}
\aliascntresetthe{exaCt}

\newaliascnt{ntnCt}{lma}
\newtheorem{ntn}[ntnCt]{Notation}
\aliascntresetthe{ntnCt}

\newcommand{\vect}[1]{\textbf{#1}}
\newcommand{\NN}{\mathbb{N}}
\newcommand{\RR}{\mathbb{R}}
\newcommand{\CC}{\mathbb{C}}
\newcommand{\KK}{\mathcal{K}}

\newcommand{\pom}{positively ordered monoid}
\newcommand{\ca}{$C^*$-al\-ge\-bra}
\newcommand{\axiomO}[1]{(O#1)}
\newcommand{\axiomW}[1]{(W#1)}

\DeclareMathOperator{\Her}{Her}
\DeclareMathOperator{\Cu}{Cu}
\DeclareMathOperator{\supp}{supp}

\newcommand{\freeVar}{\_\,}

\usepackage{stmaryrd}
\newcommand{\ihom}[1]{\llbracket #1 \rrbracket}
\newcommand{\forget}{\mathcal{F}}

\newcommand{\CatPom}{\mathrm{PoM}}

\newcommand{\CatPomProd}{\CatPom\text{-}\!\prod}

\newcommand{\CatPominvLim}{\CatPom\text{-}\!\varprojlim}

\newcommand{\CatCa}{C^*}

\newcommand{\CatCu}{\ensuremath{\mathrm{Cu}}}
\newcommand{\CatCuMor}{\CatCu}

\newcommand{\CatCuLim}{\CatCu\text{-}\!\varinjlim}
\newcommand{\CatCuProd}{\prod}

\newcommand{\CatCuCoprod}{\bigoplus}
\newcommand{\CatCuinvLim}{\CatCu\text{-}\!\varprojlim}

\newcommand{\CatCuScaled}{\ensuremath{\mathrm{Cu}_{\mathrm{sc}}}}
\newcommand{\CatCuScaledMor}{\CatCuScaled}

\newcommand{\CatC}{\mathcal{C}}

\newcommand{\CatD}{\mathcal{D}}

\newcommand{\CatQ}{\mathcal{Q}}
\newcommand{\CatQMor}{\CatQ}
\newcommand{\CatQProd}{\CatQ\text{-}\!\prod}
\newcommand{\CatQinvLim}{\CatQ\text{-}\!\varprojlim}

\newcommand{\CatW}{\mathrm{W}}
\newcommand{\CatWMor}{\CatW}

\newcommand{\CatWLim}{\CatW\text{-}\!\varinjlim}
\newcommand{\CatWSum}{\CatW\text{-}\!\bigoplus}

\newcommand{\NNbar}{\overline{\mathbb{N}}}

\newcommand{\andSep}{\,\,\,\text{ and }\,\,\,}
\newcommand{\stHom}{$^*$-ho\-mo\-morphism}
\newcommand{\CuSgp}{$\CatCu$-sem\-i\-group}

\newcommand{\CuMor}{$\CatCu$-mor\-phism}
\newcommand{\WSgp}{$\CatW$-semigroup}
\newcommand{\WMor}{$\CatW$-morphism}
\newcommand{\QSgp}{$\CatQ$-semigroup}
\newcommand{\QMor}{$\CatQ$-morphism}
\newcommand{\Paths}{\mathcal{P}}
\newcommand{\llpw}{\ll_{\mathrm{pw}}}
\newcommand{\filter}{\mathcal{U}}
\newcommand{\Cech}{\v{C}ech}

\newcommand{\cc}{\mathrm{c}}

\begin{document}

\title{Cuntz semigroups of ultraproduct \texorpdfstring{$C^*$-algebras}{C*-algebras}}

\author{Ramon Antoine}
\author{Francesc Perera}
\author{Hannes Thiel}

\date{\today}

\address{
R.~Antoine \& F. Perera, Departament de Matem\`{a}tiques,
Universitat Aut\`{o}noma de Barcelona,
08193 Bellaterra, Barcelona, Spain, and \\
Barcelona Graduate School of Mathematics (BGSMath)}
\email[]{ramon@mat.uab.cat \& perera@mat.uab.cat}

%\address{
%F.~Perera, Departament de Matem\`{a}tiques,
%Universitat Aut\`{o}noma de Barcelona,
%08193 Bellaterra, Barcelona, Spain}
%\email[]{perera@mat.uab.cat}

\address{
H.~Thiel,
Mathematisches Institut,
Universit\"at M\"unster,
Einsteinstrasse 62, 48149 M\"unster, Germany}
\email[]{hannes.thiel@uni-muenster.de}



%\thanks{}

\begin{abstract}
We prove that the category of abstract Cuntz semigroups is bicomplete.
As a consequence, the category admits products and ultraproducts.
We further show that the scaled Cuntz semigroup of the (ultra)product of a family of \ca{s} agrees with the (ultra)product of the scaled  Cuntz semigroups of the involved \ca{s}.

As applications of our results, we compute the non-stable K-Theory of general (ultra)products of \ca{s} and we characterize when ultraproducts are simple. We also give criteria that determine order properties of these objects, such as almost unperforation.
\end{abstract}

\maketitle

%==========================================================================================
%==========================================================================================
\section{Introduction}
\label{sec:intro}

%==========================================================================================
Product and ultraproduct constructions are important tools in mathematics. For \ca{s}, they are technical devices in which one can suitably encompass notions that have come to the forefront of the theory.
Most notably, ultrapoducts of \ca{s} appear in aspects of the Toms-Winter conjecture;
see, for example, \cite{BBSTWW19}.
They also play an important role in detecting tensorial absorption by the Jiang-Su algebra for \ca{s} with finite nuclear dimension;
see \cite{Kir06CentralSeqPI} (and also \cite{RobTik17NucDimNonSimple}).

The Cuntz semigroup is an invariant for \ca{s}, first introduced in \cite{Cun78DimFct} with the aim of proving that simple, stably finite, \ca{s} admit quasitraces.
Due to the fact that it contains a great deal of information about the \ca{}, its uses have been proved to be far-reaching in subsequent works.
For example, the Cuntz semigroup has become key to distinguishing \ca{s} that fail to satisfy regularity conditions and thus are not amenable to classification using the Elliott invariant (see, for example, \cite{Ror03FinInfProj} and \cite{Tom08ClassificationNuclear}).
Further, for the class of simple, separable, nuclear \ca{s} that absorb the Jiang-Su algebra tensorially, it was proved in \cite{AntDadPerSan14RecoverElliott} that the Cuntz semigroup after tensoring with the circle is functorially equivalent to the Elliott invariant.

The categorical framework needed to analyse the Cuntz semigroup as an invariant was developed in the pioneering work by Coward, Elliott and Ivanescu in \cite{CowEllIva08CuInv}.
They defined the category $\CatCu$ of abstract Cuntz semigroups (also called \CuSgp{s}) and showed that the process of assigning to each \ca{} $A$ the (concrete) Cuntz semigroup of its stabilisation is a sequentially continuous functor.
The properties of the category $\Cu$ and its relation with the category of \ca{s} was further developed in \cite{AntPerThi18TensorProdCu}, \cite{AntPerThi17arX:AbsBivariantCu}, and \cite{AntPerThi120AbsBivarII}.
It is therefore interesting to seek efficient methods to compute the Cuntz semigroup of products and ultraproducts.
This is one of the main objectives of this paper.

Buiding on previous work (as explained below), we summarize the main structural features of the category $\CatCu$ as follows:

%==========================================================================================
\begin{thm*}
The category $\CatCu$ of abstract Cuntz semigroups is a closed, symmetric, monoidal, bicomplete category.
\end{thm*}

%==========================================================================================
That $\CatCu$ is symmetric and monoidal was established in \cite{AntPerThi18TensorProdCu};
on the other hand, it was shown in \cite{AntPerThi17arX:AbsBivariantCu} that $\CatCu$ is closed.
We prove in this paper that $\CatCu$ is bicomplete, that is, complete and cocomplete (see Theorems~\ref{prp:CuComplete} and~\ref{prp:CuCocomplete} below). This means that $\CatCu$ admits arbitrary small limits and colimits, and thus arbitrary products and inductive limits. (The latter was already proved in \cite{AntPerThi18TensorProdCu}.) 

In order to show that $\CatCu$ is complete, we take advantage of the fact that it sits coreflectively in a bigger category, termed $\CatQ$ in \cite{AntPerThi17arX:AbsBivariantCu}.
It follows that the product of a family of $\CatCu$-semigroups is not just their set-theoretic product, but rather a completion of the latter by means of the $\tau$-construction (see \autoref{pgr:tauconstr} for more details).
Roughly speaking, the $\tau$-construction consists of considering certain equivalence classes of continuous paths indexed on the interval $(-\infty, 0]$.

We further apply our results to \ca{s}. To this end, we first study to what extent the Cuntz semigroup functor preserves these constructions. The fact that the Cuntz semigroup functor is stable by definition introduces a difficulty when dealing with products and ultraproducts of arbitrary infinite families of \ca{s}, since these are not stable even if the algebras are. In order to circumvent this difficulty, we need a convenient substitute for stability. This is found in property~(S), as introduced by Hjelmborg and R{\o}rdam in \cite{HjeRor98Stability}. 

Another important ingredient in our discussion is the notion of scale. For a \ca{} $A$, the natural scale of $\Cu(A)$ captures the information of equivalence classes coming from positive elements in $A$.
This was first considered in \cite{PerTomWhiWin14CuStabilityCloseCa}. We define in \autoref{sec:scaled} scales in general abstract $\CatCu$-semigroups and introduce the category $\CatCuScaled$, in which $\CatCu$ sits reflexively. This is the natural categorical framework in which Cuntz semigroups of products and ultraproducts of arbitrary families of \ca{s} can be computed.

Ultraproducts over an ultrafilter $\filter$ are introduced categorically in \autoref{pgr:ultraprCat} as inductive limits of products over sets in $\filter$ following, for example, \cite{Ekl77UltraprodAlg}. This becomes relevant because, as we show later, the scaled Cuntz semigroup functor preserves arbitrary inductive limits and products. We prove:

%==========================================================================================
\begin{thm*}[{Theorems~\ref{prp:CuScaledComplete},~\ref{thm:CuProd} and~\ref{thm:CuUltraprod}}]
The category $\CatCuScaled$ is bicomplete, and the scaled Cuntz semigroup functor preserves arbitrary products and ultraproducts of \ca{s}.
\end{thm*}

%==========================================================================================
More precisely, given a family $(A_j)_{j\in J}$ of \ca{s} and an ultrafilter $\filter$ over~$J$, there is a natural commutative diagram
\[
\xymatrix@R-10pt{
\Cu(\prod_{j\in J}A_j) \ar[r]^{\Phi} \ar@{->>}[d]%_{\Cu(\pi_\filter)}
& \CatCuProd_{j\in J} \Cu(A_j) \ar@{->>}[d] \\
\Cu(\prod_\filter A_j) \ar[r]^{\Phi_\filter} & \prod_\filter\Cu(A_j)
}.
\]
where the horizontal maps are order-embeddings whose images are ideals.
These maps induce order-isomorphisms between scaled Cuntz semigroups.
In the unital setting, the said isomorphisms can be viewed as
\[
\Cu(\prod_j A_j) \cong \big\{ x\in \prod_j\Cu (A_j)\mid x\leq \infty u \big\},
\]
where $u$ is the image of the class of the unit via the map $\Phi$ and $\infty u=\sup_n nu$.
A similar description is obtained for the ultraproduct.

Since the categorical picture of an ultraproduct is not the most commonly used in the category of \ca{s}, we make the connection explicit in \autoref{pgr:ultraprCAlg} to also view an ultraproduct of \ca{s} as the quotient of their product modulo the ideal of strings that converge to zero along the ultrafilter.
This is also a convenient way to compute its Cuntz semigroup, as a quotient of the Cuntz semigroup of the product of the algebras modulo a suitable ideal, and it becomes a useful way to deal with ultraproducts in concrete examples (see \autoref{prp:CuUltraprodConcrete} and \autoref{prp:prodE0}).

As applications of our results, we compute the monoid of Murray-von Neumann classes of projections of a product and ultraproduct of a family of stably finite \ca{s}, recovering and completing previous work of Li (\cite{Li05Ultraproducts}). We also characterize in \autoref{prp:charSimpleUltraProdCa} when an ultraproduct of \ca{s} over a countably incomplete ultrafilter is simple; this unifies previous results due to Kirchberg (see \cite{Ror02Classification}) and also results from \cite{FarHarLupRobTikVigWin06arX:ModelThy}.

%==========================================================================================
%==========================================================================================
\section*{Acknowledgements}
Part of this research was conducted during stays at the CRM (Barcelona), and the ICMAT (Madrid).
The authors would like to thank both institutions for financial support and for providing inspiring working conditions.

The authors would also like to thank Leonel Robert and G\'abor Szab\'o for valuable conversations and feedback.

The two first named authors were partially supported by MINECO (grant No.\ MTM2017-83487-P), and by the Comissionat per Universitats i Recerca de la Generalitat de Catalunya (grant No.\ 2017-SGR-1725).~The third named author was partially supported by the Deutsche Forschungsgemeinschaft (DFG, German Research Foundation) under the SFB 878 (Groups, Geometry \& Actions) and under Germany's Excellence Strategy EXC 2044-390685587 (Mathematics M\"{u}nster: Dynamics-Geometry-Structure).

We wish to thank the referee for numerous comments that have helped to improve the manuscript.

%==========================================================================================
%==========================================================================================
\section{Preliminaries: Categorical Framework}
\label{sec:basics}

%==========================================================================================
\subsection{The category \texorpdfstring{$\CatCu$}{Cu} of abstract Cuntz semigroups}

%==========================================================================================
In this subsection, we recall the definition of the category $\CatCu$ as introduced in \cite{CowEllIva08CuInv}.

%==========================================================================================
\begin{pgr}
A \emph{\pom} is a commutative monoid $M$, written additively with neutral element $0$, equipped with a translation invariant partial order $\leq$ for which $0$ is the smallest element.
The category of \pom{s} will be denoted by $\CatPom$, where the morphisms are maps that preserve addition, order, and the zero element.
See \cite[Appendix~B.2]{AntPerThi18TensorProdCu} for details.
\end{pgr}

%==========================================================================================
\begin{pgr}
\label{pgr:auxiliary}
Let $(X,\leq)$ be a partially ordered set.
A binary relation $\prec$ on $X$ is called an \emph{auxiliary relation} (see \cite[Definition~I-1.11, p.57]{GieHof+03Domains}) if:
\begin{itemize}
\item[{\rm (i)}]
If $x\prec y$ then $x\leq y$, for all $x,y\in X$.
\item[{\rm (ii)}]
If $w\leq x\prec y\leq z$ then $w\prec z$, for all $w,x,y,z\in X$.
\end{itemize}
If, further, $X$ is a monoid, then an auxiliary relation $\prec$ is said to be \emph{additive} if it is compatible with addition and $0\prec x$ for every $x\in X$. Observe that an auxiliary relation is automatically transitive: if $x\prec y$ and $y\prec z$, then by condition (i) $x\leq y$ and applying condition (ii) to $x\leq y\prec z\leq z$, we obtain $x\prec z$.

Of particular importance for us is the (sequential version of the) \emph{compact containment relation}, also called \emph{way-below relation}.
Given $x,y\in X$, we say that $x$ is \emph{compactly contained} in $y$, or that $x$ is \emph{way-below} $y$, in symbols $x\ll y$, provided that for every increasing sequence $(y_n)_n$ such that the supremum $\sup_n y_n$ exists and satisfies $y\leq\sup_n y_n$, there exists $k\in\NN$ such that $x\leq y_k$.
\end{pgr}

%==========================================================================================
\begin{dfn}[\cite{CowEllIva08CuInv}]
\label{dfn:CatCu}
A \emph{\CuSgp}, also called an \emph{abstract Cuntz semigroup}, is a \pom{} $S$ that satisfies the following axioms:
\begin{itemize}
\item[\axiomO{1}]
Every increasing sequence $(x_n)_n$ in $S$ has a supremum $\sup_n x_n\in S$.
\item[\axiomO{2}]
Every element of $S$ is the supremum of a $\ll$-increasing sequence.
\item[\axiomO{3}]
If $x'\ll x$ and $y'\ll y$ then $x'+y'\ll x+y$, for all $x',y',x,y\in S$.
\item[\axiomO{4}]
If $(x_n)_n$ and $(y_n)_n$ are increasing sequences in $S$, then $\sup_n(x_n+y_n)=\sup_n x_n+\sup_n y_n$.
\end{itemize}
A \emph{\CuMor} is a $\CatPom$-morphism that preserves the way-below relation and suprema of increasing sequences.
Given \CuSgp{s} $S$ and $T$, the set of all \CuMor{s} from $S$ to $T$ is denoted by $\CatCuMor(S,T)$.

An \emph{ideal} in a \CuSgp{} $S$ is a submonoid $J\subseteq S$ that is downward hereditary (if $x,y\in S$ satisfy $x\leq y\in J$, then $x\in J$) and closed under passing to suprema of increasing sequences.
\end{dfn}

%==========================================================================================
\begin{pgr}
\label{pgr:CuA}
Let $A$ be a \ca{}, and let $a, b$ be positive elements in the stabilization $A\otimes\KK$.
Recall that $a$ is \emph{Cuntz subequivalent} to $b$, in symbols $a\precsim b$, provided that there is a sequence $(x_n)_n$ in $A\otimes\KK$ such that $x_nbx_n^*$ converges to $a$ in norm.
In case that $a\precsim b$ and $b\precsim a$, then $a$ and $b$ are said to be \emph{Cuntz equivalent}, and we write $a\sim b$.
Then $\sim$ is an equivalence relation on $(A\otimes\KK)_+$.
The \emph{Cuntz semigroup} of $A$ is the quotient set
\[
\Cu(A) :=(A\otimes\mathcal K)_+/\!\!\sim.
\]
The equivalence class of an element $a\in(A\otimes\KK)_+$ in $\Cu(A)$ is denoted by $[a]$.
Addition on $\Cu(A)$ is defined by $[a]+[b]=[\left(\begin{smallmatrix} a & 0 \\ 0 & b \end{smallmatrix}\right)]$, for $a,b\in (A\otimes\KK)_+$ (using an isomorphism $M_2(\KK)\cong\KK$).
The neutral element for this operation is the class of the zero element in $A\otimes\KK$.
Finally, we define a partial order on $\Cu(A)$ by setting $[a]\leq [b]$ provided $a\precsim b$.
In this way $\Cu(A)$ becomes a \pom{}.

It was proved in \cite{CowEllIva08CuInv} that $\Cu(A)$ is a $\CatCu$-semigroup.
Moreover, a \stHom{} $\varphi\colon A\to B$ induces a $\CatCu$-morphism $\Cu(\varphi)\colon\Cu(A)\to\Cu(B)$ by $\Cu(\varphi)([a])=[(\varphi\otimes\mathrm{id}_{\KK})(a)]$, for $a\in(A\otimes\KK)_+$.
This defines a functor $\Cu(\freeVar)$, from the category $\CatCa$ of \ca{s} with \stHom{s} to the category $\CatCu$.

It was also shown in \cite{CowEllIva08CuInv} that the category $\CatCu$ admits sequential inductive limits and that the functor $\Cu(\freeVar)$ is sequentially continuous, that is, it preserves sequential inductive limits.
This was extended to inductive limits over arbitrary directed sets in \cite[Corollary~3.2.9]{AntPerThi18TensorProdCu}.
In \autoref{prp:CuCocomplete} we show that $\CatCu$ is even a cocomplete category, that is, it admits arbitrary (small) colimits. 

We will often use the following result, which is called `R{\o}rdam's lemma' (see \cite[Proposition~2.4]{Ror92StructureUHF2} or \cite[Theorem~2.30]{Thi17:CuLectureNotes}):
Let $a,b\in A_+$.
Then the following are equivalent:
\begin{enumerate}
\item
$a\precsim b$.
\item
For every $\varepsilon>0$ there exists $\delta>0$ such that $(a-\varepsilon)_+\precsim(b-\delta)_+$.
\item
For every $\varepsilon>0$ there exist $\delta>0$ and $s\in A$ such that $(a-\varepsilon)_+=ss^*$ and $s^*s\in\Her((b-\delta)_+)$.
\end{enumerate}
\end{pgr}

%==========================================================================================
%==========================================================================================
\subsection{The category \texorpdfstring{$\CatW$}{W}}

%==========================================================================================
In this subsection, we recall the definition of the category $\CatW$ from \cite{AntPerThi18TensorProdCu}, with a slight deviation;
see \autoref{rmk:CatWDifference}.

%==========================================================================================
\begin{dfn}
\label{dfn:CatW}
A \emph{\WSgp} is a commutative monoid $S$ together with a transitive binary relation $\prec$ satisfying $0\prec x$ for every $x\in S$, and such that the following conditions are satisfied:
\begin{enumerate}
\item[\axiomW{1}]
For each $x\in S$, there is a $\prec$-increasing sequence $(x_n)_{n}$ with $x_n\prec x$ for each $n$, and such that for any $y\prec x$ there is $n\in\NN$ satisfying $y\prec x_n$.
\item[\axiomW{3}]
If $x'\prec x$ and $y'\prec y$, then $x'+y'\prec x+y$, for all $x',y',x,y\in S$.
\item[\axiomW{4}]
If $x,y,z\in S$ satisfy $x\prec y+z$, then there are $y',z'\in S$ such that $y'\prec y$, $z'\prec z$, and $x\prec y'+z'$.
\end{enumerate}
Given \WSgp{s} $S$ and $T$, a \emph{\WMor} is a map $f\colon S\to T$ that preserves addition, the relation $\prec$, the zero element, and that is continuous in the sense that for every $x\in S$ and $y\in T$ satisfying $y\prec f(x)$ there exists $x'\in S$ such that $x'\prec x$ and $y\prec f(x')$.
The set of all \WMor{s} from $S$ to $T$ is denoted by $\CatWMor(S,T)$.
\end{dfn}

%==========================================================================================
\begin{rmk}
\label{rmk:CatWDifference}
In \cite[Definition~2.1.1]{AntPerThi18TensorProdCu}, we gave a slightly different definition of a \WSgp.
To clarify the change from \autoref{dfn:CatW}, let us denote the category defined in \cite{AntPerThi18TensorProdCu} by $\CatW^+$.
Thus, a $\CatW^+$-semigroup is a \pom~$S$ together with an auxiliary relation $\prec$ satisfying $0\prec x$ for every $x\in S$, satisfying \axiomW{1}, \axiomW{3} and \axiomW{4} as in \autoref{dfn:CatW}, and satisfying \axiomW{2} which requires that each $x$ in $S$ is the supremum of $x^\prec:=\{x':x'\prec x\}$.
Further, a $\CatW^+$-morphism is an order-preserving monoid homomorphism that satisfies the continuity condition from \autoref{dfn:CatW}.

If $S$ is a $\CatW^+$-semigroup, then by forgetting the partial order on $S$ we obtain a \WSgp.
Further, every $\CatW^+$-morphism naturally induces a \WMor{}.
This assignment defines a `forgetful' functor $F\colon\CatW^+\to\CatW$.
	
Conversely, let $(S,\prec)$ be a \WSgp.
We define a binary relation on $S$ by setting $x\leq y$ if and only if $x^{\prec}\subseteq y^{\prec}$.
Then $\leq$ is a translation-invariant pre-order.
Now, for $x,y\in S$, we set $x\prec_+ y$ if and only if there exists $y'\in S$ such that $x\leq y'$ and $y'\prec y$.
Given $w,x,y,z\in S$, we have that $x\prec y$ implies $x\prec_+ y$, also that $x\prec_+ y$ implies $x\leq y$, and finally we have that $w\leq x\prec_+ y\leq z$ implies $w\prec_+ z$.
Let us check the latter.
If $x\prec_+ y$, then there is $y'\in S$ with $x\leq y'\prec y$, and it follows that $w\leq y'\prec y$.
Since $y'\prec y$ and $y\leq z$, we have $y'\prec z$, and therefore $w\prec_+ z$.
It is easy to check that $(S,\prec_+)$ satisfies \axiomW{1}, \axiomW{3}, and \axiomW{4}.
	
Given $x\in S$, let us verify $x=\sup x^{\prec_+}$.
Clearly, $x$ is an upper bound for $x^{\prec_+}$.
To show that $x$ is the smallest upper bound, let $y\in S$ be another upper bound for $x^{\prec_+}$.
Given $w\in S$ with $w\prec x$, we can apply \axiomW{1} to find $w'\in S$ such that $w\prec w'\prec x$, and in particular $w'\in x^{\prec_+}$.
By assumption we have $w'\leq y$ and, since $w\prec w'$, we obtain $w\prec y$.
This implies $x\leq y$.

Thus, $(S,\leq,\prec_+)$ is a $\CatW^+$-semigroup if and only if the pre-order $\leq$ is a \emph{partial order}.
By antisymmetrization applied to the preordered set $(S,\leq)$,
%, to a pre-ordered set $(M,\leq)$ one can naturally associate a partially ordered set $(M',\leq')$.
%Applying this to $(S,\leq)$,
one can naturally associate a partially ordered set $(S',\leq')$, and
one can show that $\prec_+$ induces an auxiliary relation $\prec_+'$ on $S'$ in such a way that $(S',\leq',\prec_+')$ is a $\CatW^+$-semigroup. This defines a functor $G\colon\CatW\to\CatW^+$ such that $G\circ F\cong\mathrm{id}$.
However, $F\circ G\ncong\mathrm{id}$.
For example, take $S=\NN$, with $x\prec y$ if and only if $x=0$.
Then $F(G(S))\cong\{0\}$, which is not isomorphic to $S$.

%We note that the assignment $(S,\prec_+)\mapsto (S,\leq,\prec)$ defines a functor $G\colon \CatW_0\to\CatW$. 
%Indeed, if $f\colon S\to T$ is a $\CatW_0$-morphism, then simply set $G(f)=f$. 
%We only need to verify that $f$ is order-preserving. 
%For then, if $x\prec y$ in $S$, there exists $y'\in S$ with $y'\prec_+ y$ such that $x\leq y'$, and thus $f(x)\leq f(y')\prec_+ f(y)$, which implies $f(x)\prec f(y)$. 
%It also follows easily that $f$ is continuous in the sense of \cite[Definition~2.1.2]{AntPerThi18TensorProdCu}.
%Now, to check that $f$ is order-preserving, assume that $x\leq y$ in $S$, and let $z\in T$ satisfy $z\prec_+ f(x)$. 
%As $f$ is continuous, there is $x'\prec_+ x$ in $S$ with $z\prec_+f(x')$. 
%Then $x'\prec_+ y$, whence $f(x')\prec_+f(y)$. Thus also $z\prec_+ f(y)$ in $T$, and this implies $f(x)\leq f(y)$.
	
In the proofs below we work exclusively in the category $\CatW$ as given by \autoref{dfn:CatW}.
Thus, we do not require that $\prec$ induces a partial order.
Moreover, by dropping this requirement, we may forget about the (partial) order altogether.
Furher, we will see below that many results from \cite{AntPerThi18TensorProdCu} remain valid for the more relaxed definition of the category $\CatW$ used here.
%For example, in \autoref{pgr:gamma} below we will have that $\gamma$ is a functor that only needs the auxiliary relation, hence it may be defined on $\CatW_0$ and be denoted also as $\gamma$ and the same arguments as the ones in \cite{AntPerThi18TensorProdCu} may be used.
%It also follows that $\gamma\circ F=\gamma$.

% 	Hence, the only difference between \autoref{dfn:CatW} and \cite[Definition~2.1.1]{AntPerThi18:TensorProdCu} is that in the latter case the auxiliary relation is required to induce a partial order instead of a pre-order.
% 	However, in the proofs it is not necessary to require that $\prec$ induces a partial order.
% 	Moreover, by dropping this requirement, we may forget about the (partial) order altogether, which further simplifies some arguments.
\end{rmk}

%==========================================================================================
\begin{pgr}
\label{pgr:CuInW}
Given a \CuSgp{} $S$, the monoid $S$ together with the relation $\ll$ is a \WSgp.
Moreover, given two \CuSgp{s} $S$ and $T$, a map $\varphi\colon S\to T$ is a \CuMor{} if and only if $\varphi$, considered as a map from the \WSgp{} $(S,\ll)$ to the \WSgp{} $(T,\ll)$, is a \WMor.
Thus, we obtain a functor $\iota\colon\CatCu\to\CatW$ that embeds $\CatCu$ as a full subcategory of $\CatW$.
\end{pgr}

%==========================================================================================
\begin{pgr}
\label{pgr:gamma}
Given a \WSgp{} $S$, let $\bar{S}$ denote the set of all $\prec$-increasing sequences in $S$.
We equip $\bar{S}$ with componentwise addition, giving it the structure of a commutative monoid.
For sequences $\textbf{x}:=(x_n)_n$ and $\textbf{y}:=(y_n)_n$ in $\bar{S}$, we write $\textbf{x}\subseteq \textbf{y}$ if for every $k\in\NN$, there is $l\in\NN$ such that $x_k\prec y_l$.
Then $\subseteq$ is a preorder on $\bar{S}$.
Set $\textbf{x}\sim\textbf{y}$ if and only if $\textbf{x}\subseteq\textbf{y}$ and $\textbf{y}\subseteq\textbf{x}$.
Then $\sim$ is an equivalence relation on $\bar{S}$.
Define
\[
\gamma(S) := \bar{S}/\!\sim.
\]
Given $\textbf{x}\in \bar{S}$, we use use $[\textbf{x}]$ to denote its class in $\gamma(S)$.
The relation $\subseteq$ induces a partial order on $\gamma(S)$, by setting $[\textbf{x}]\leq [\textbf{y}]$ if and only if $\textbf{x}\subseteq\textbf{y}$.

The proof of \cite[Proposition~3.1.6]{AntPerThi18TensorProdCu} shows that $\gamma(S)$ is a $\CatCu$-semigroup, which we call the \emph{$\CatCu$-completion} of $S$.
Further, given a \WMor{} $\varphi\colon S\to T$, we define $\gamma(\varphi)\colon\gamma(S)\to\gamma(T)$ by $\gamma(\varphi)([(x_n)_n]):=[(\varphi(x_n))_n]$, for $(x_n)_n\in\bar{S}$.
One can show that $\gamma(\varphi)$ is a \CuMor.
This defines a covariant functor $\gamma\colon\CatW\to\CatCu$.

Let $S$ be a \WSgp{}.
We define $\alpha_S\colon S\to\gamma(S)$ by $\alpha_S(x):=[(x_n)_n]$, for $x\in S$, where $(x_n)_n$ is a $\prec$-increasing sequence obtained by applying \axiomW{1} to $x$.
One verifies that $\alpha_S$ is a well-defined \WMor.
Adapting the proof of \cite[Theorem~3.1.10]{AntPerThi18TensorProdCu}, we obtain the following result.
\end{pgr}

%==========================================================================================
Recall that a full subcategory is said to be \emph{reflective} if the inclusion functor has a left adjoint. 

%==========================================================================================
\begin{thm}
\label{prp:CuReflW}
The category $\CatCu$ is a full, reflective subcategory of $\CatW$.
The functor $\gamma\colon\CatW\to\CatCu$ is a left adjoint to the inclusion functor $\iota\colon\CatCu\to\CatW$.
Moreover, given a \WSgp{} $S$, the map $\alpha_S\colon S\to\gamma(S)$ is a universal \WMor:
For every \CuSgp{} $T$, there is a natural bijection
\[
\CatCuMor\big( \gamma(S), T \big) \cong \CatWMor\big( S, \iota(T) \big),
\]
implemented by sending a \CuMor{} $\varphi\colon\gamma(S)\to T$ to $\varphi\circ\alpha_S$.
\end{thm}

%==========================================================================================
%==========================================================================================
\subsection{The category \texorpdfstring{$\CatQ$}{Q}}

%==========================================================================================
In this subsection, we recall the definition of the category $\CatQ$ and of the $\tau$-construction, as introduced in \cite{AntPerThi17arX:AbsBivariantCu}.
Since we will apply the $\tau$-construction only to $\CatQ$-semigroups, we use a simplified definition of path that is different from that used in \cite{AntPerThi17arX:AbsBivariantCu}.
We remark that the $\tau$-construction can be applied to a much more general class of semigroups equipped with an auxiliary relation by using a more general concept of paths;
see \cite[Definition~3.3]{AntPerThi17arX:AbsBivariantCu}.

%==========================================================================================
\begin{dfn}
\label{dfn:CatQ}
A \emph{\QSgp} is a \pom{} satisfying axioms~\axiomO{1} and \axiomO{4} from \autoref{dfn:CatCu}, together with an additive auxiliary relation $\prec$.
A \emph{\QMor} is a $\CatPom$-morphism that preserves the auxiliary relation and suprema of increasing sequences.
Given \QSgp{s} $S$ and $T$, the set of all \QMor{s} from $S$ to $T$ is denoted by $\CatQMor(S,T)$.
\end{dfn}

%==========================================================
\begin{pgr}
\label{pgr:CuInQ}
Given a \CuSgp{} $S$, then $S$ together with the relation $\ll$ is a \QSgp.
Moreover, given two \CuSgp{s} $S$ and $T$, a map $\varphi\colon S\to T$ is a \CuMor{} if and only if it is a \QMor{} when considered as a map from $(S,\ll)$ to $(T,\ll)$.
We obtain a functor $\iota\colon\CatCu\to\CatQ$ that sends a \CuSgp{} $S$ to the \QSgp{} $(S,\ll)$, and embeds $\CatCu$ as a full subcategory of $\CatQ$.
\end{pgr}

%==========================================================================================
\begin{pgr}
\label{pgr:tauconstr}
Let $S=(S,\prec)$ be a $\CatQ$-semigroup.
A \emph{path} in $S$ is an order-preserving map $f\colon (-\infty,0]\to S$ such that $f(t)=\sup_{t'<t}f(t')$ for all $t\in(-\infty,0]$, and such that $f(t')\prec f(t)$ whenever $t'<t$.
We denote the set of paths in $S$ by $\Paths(S)$.

For $f,g\in \Paths(S)$ we define $f+g$ by $(f+g)(t)=f(t)+g(t)$. Together with the constant zero path, this gives $\Paths(S)$ the structure of a commutative monoid.

Given $f,g\in \Paths(S)$, we write $f\precsim g$ if, for every $t<0$, there is $t'<0$ such that $f(t)\prec g(t')$.
Set $f\sim g$ if $f\precsim g$ and $g\precsim f$.
It follows from \cite[Lemma~3.4]{AntPerThi17arX:AbsBivariantCu} that the relation $\precsim$ is reflexive, transitive, and compatible with the addition of paths.
We set
\[
\tau(S) := \Paths(S)/\!\sim.
\]
Given $f\in \Paths(S)$, its equivalence class in $\tau(S)$ is denoted by $[f]$.
Equip $\tau(S)$ with an addition and order by setting $[f]+[g]:=[f+g]$ and $[f]\leq [g]$ provided $f\precsim g$.
It is proved in \cite[Theorem~3.15]{AntPerThi17arX:AbsBivariantCu} that $\tau(S)$ is a \CuSgp.

We define $\varphi_S\colon\tau(S)\to S$ by $\varphi_S([f]) := f(0)$, for $f\in\Paths(S)$.
One verifies that $\varphi_S$ is a well-defined \QMor.
Given a path $f\in\Paths(S)$, we think of $f(0)$ as the endpoint of $f$.
Accordingly, we call $\varphi_S$ the \emph{endpoint map}.

The construction just outlined will be referred to as the \emph{$\tau$-construction}.
Given a $\CatQ$-morphism $\varphi\colon S\to T$, we define $\tau(\varphi)\colon\tau(S)\to\tau(T)$ by $\tau(\varphi)([f]):=[\varphi\circ f]$, for $f\in \Paths(S)$.
One can show that $\tau(\varphi)$ is a \CuMor.
This defines a covariant functor $\tau\colon\CatQ\to\CatCu$.
\end{pgr}

%===================================================
\begin{thm}[{\cite[Theorem~4.12]{AntPerThi17arX:AbsBivariantCu}}]
\label{prp:CuCoreflQ}
The category $\CatCu$ is a full, coreflective subcategory of $\CatQ$.
The functor $\tau\colon\CatQ\to\CatCu$ is a right adjoint to the inclusion functor $\iota\colon\CatCu\to\CatQ$.
Moreover, given a \QSgp{} $S$, the endpoint map $\varphi_S\colon\tau(S)\to S$ is a universal \QMor:
For every \CuSgp{} $T$, there is a natural bijection
\[
\CatCuMor\big( T,\tau(S) \big) \cong \CatQMor\big( \iota(T),S \big),
\]
implemented by sending a \CuMor{} $\psi\colon T\to\tau(S)$ to $\varphi_S\circ\psi$.
\end{thm}

%==========================================================================================
%==========================================================================================
\section{Bicompleteness of \texorpdfstring{$\CatCu$}{Cu}}
\label{sec:productsandinverselims}

%==========================================================================================
In this section we show that the category $\CatCu$ is bicomplete, that is, complete and cocomplete.
To achieve this, we use that $\CatCu$ is a full, coreflective subcategory of $\CatQ$ and also a full, reflective subcategory of $\CatW$;
see Theorems~\ref{prp:CuCoreflQ} and~\ref{prp:CuReflW}.

We begin by briefly recalling the notions of small (co)limits and (co)completeness in arbitrary categories;
see, for example, \cite[Section~2.6]{Bor94HandbookCat1}.

%==========================================================================================
\begin{pgr}[Small limits and colimits]
\label{pgr:smallcomplete}
Let $I$ be a small category, and let $F\colon I\to\CatC$ be a functor to a category $\CatC$.
Recall that a \emph{cone} to $F$ is a pair $(L,\varphi)$, where $L$ is an object in $\CatC$ and $\varphi=(\varphi_i)_{i\in I}$ is a collection of $\CatC$-morphisms $\varphi_i\colon L\to F(i)$ with the property that, whenever $f\colon i\to j$ is a morphism in $I$, we have $F(f)\circ\varphi_i=\varphi_j$.
This means that the left diagram below commutes.

By definition, a \emph{limit} of $F$ is a cone $(L,\varphi)$ that is universal in the sense that every cone $(L',\varphi')$ factors through $(L,\varphi)$, that is, there exists a unique $\CatC$-morphism $\alpha\colon L'\to L$ such that $\varphi_i'=\varphi_i\circ\alpha$ for every $i$.
This means that the right diagram below commutes.
\[
\xymatrix@R-20pt{
& F(i) \ar[dd]^{F(f)}
& &  &  & F(i) \ar[dd]^{F(f)} \\
L \ar[ur]^{\varphi_i} \ar[dr]_{\varphi_j}
&  & & L' \ar[r]^\alpha \ar@/^0.8pc/[urr]^{\varphi_i'} \ar@/_0.8pc/[drr]_{\varphi_j'}
& L \ar[ur]^{\varphi_i} \ar[dr]_{\varphi_j}  &   \\
& F(j) & & & & F(j) \\
}
\]

If a limit of $F$ exists then it is unique up to natural isomorphism and we shall denote it by  $(\CatC\text{-}\!\varprojlim F,\pi)$ or just by $\CatC\text{-}\!\varprojlim F$.

The notion of \emph{cocone} to a functor $F\colon I\to \mathcal C$ is  defined dually, as a pair $(M,\psi)$ with $M$ an object of $\CatC$ and $\psi=(\psi_i)_{i\in I}$ a collection of $\CatC$-morphisms $\psi_i\colon F(i)\to M$ with dual compatibility relations.
A \emph{colimit} of $F$ is then defined as a universal cocone. 
If a colimit of $F$ exists then it is unique up to natural isomorphism and we shall denote it by $(\CatC\text{-}\varinjlim F,\sigma)$, or simply by $\CatC\text{-}\varinjlim F$.

The category $\CatC$ is termed \emph{complete} (\emph{cocomplete}) if all small limits (colimits) exist, that is, if every functor from a small category to $\CatC$ has a limit (colimit).
\end{pgr}

%==========================================================================================
\begin{pgr}
Let $F\colon I\to\CatC$ be a functor from a small category $I$ to a category $\CatC$.
Recall that, if $I$ has no morphisms (except the identities at each object), then $F$ corresponds to a selection of a family $(X_i)_{i\in I}$ of objects in $\CatC$. 
In this case, a limit of $F$ is called a \emph{product} of the family $(X_i)_{i\in I}$ and is denoted by $\prod_{i\in I}X_i$. 
Dually we define a \emph{coproduct} which we denote by $\coprod_{\in I}X_i$.

In case $I$ is an upward directed, partially ordered set, then the functor $F$ corresponds to a family $(X_i)_{i\in I}$ of objects in $\CatC$ together with compatible $\CatC$-morphisms $\varphi_{j,i}\colon X_i\to X_j$ whenever $i\leq j$.
A colimit of $F$ is called an \emph{inductive limit} or \emph{direct limit}.
Inverse limits are defined similarly.

Finally, if $I=\{0,1\}$ is the category with two objects and two morphisms $0\to 1$, then $F\colon I\to\CatC$ corresponds to two $\CatC$-morphisms $\varphi,\psi\colon X\to Y$. 
A colimit of $F$ is then called a \emph{coequalizer}.
Thus, a coequalizer for $(\varphi,\psi)$ is a universal pair $(C,\eta)$, where $C$ is an object of $\CatC$ and $\eta\colon Y\to C$ is a $\CatC$-morphism such that $\eta\circ\varphi=\eta\circ\psi$.
\end{pgr}

%==========================================================================================
%==========================================================================================
\subsection{Completeness}

%==========================================================================================
In this subsection we prove that $\CatCu$ is a complete category;
see \autoref{prp:CuComplete}.
In particular, $\CatCu$ has products over infinite index sets. 
We will prove in the sequel that the scaled Cuntz semigroup functor preserves (ultra)products of \ca{s}; 
see \autoref{thm:CuProd} and \autoref{thm:CuUltraprod}.

%==========================================================================================
\begin{pgr}
\label{pgr:prodinvlim}
Let $I$ be a set.
Let $(S_i)_{i\in I}$ be a family of \pom{s}.
Define
\[
\CatPomProd_{i\in I}S_i
:= \big\{ (s_i)_{i\in I} : s_i \in S_i\text{ for all }i\in I \big\}\,.
\]
Equipped with componentwise order and addition, $\CatPomProd_{i\in I}S_i$ is a \pom{}.
Given $i\in I$, we let $\pi_i\colon\CatPomProd_i S_i\to S_i$ be the projection map given by $\pi_i\big( (s_j)_{j\in I} \big) := s_i$.
It is easy to verify that $\left(\CatPomProd_{i\in I}S_i,(\pi_i)_i\right)$ is the product of the family $(S_i)_i$ in $\CatPom$.

Next, let $I$ be a small category and let $F\colon I\to\CatPom$ be a functor.
Set
\begin{align}
\label{pgr:prodinvlim:eq2}
S:= \left\{ (s_i)_{i\in I}\in\CatPomProd_{i\in I} F(i) :  F(f)(s_i)=s_j \text{ for all } f\colon i\to j \text{ in }I \right\}.
\end{align}
It is straightforward to verify that $0\in S$ and that $S$ is closed under addition in $\CatPomProd_i F(i)$.
Thus, $S$ inherits the structure of a \pom{}. 

For each $i\in I$, the projection map $\pi_i\colon\CatPomProd_j F(j)\to F(i)$ restricts to a $\CatPom$-morphism $\pi_i\colon S\to F(i)$.
It is straightforward to verify that $(S,\pi)$, where $\pi=(\pi_i)_{i\in I}$, is the limit of $F$ in $\CatPom$.
(Observe that $\CatPom\text{-}\!\varprojlim F$ agrees, as a set, with the limit of the functor $F$ taking as target category the category of sets.)
\end{pgr}

%==========================================================================================
\begin{lma}
\label{prp:limitPomO1O4}
Let $I$ be a small category, let $F\colon I\to\CatPom$ be a functor, and let $(S,\pi)$ be the limit of $F$ in $\CatPom$.
Assume that $F(i)$ satisfies \axiomO{1} and \axiomO{4}, for each $i\in I$, and assume that $F(f)\colon F(i)\to F(j)$ preserves suprema of increasing sequences for each morphism $f\colon i\to j$ in $I$.
Then $S$ satisfies \axiomO{1} and \axiomO{4}, and for each $i\in I$ the projection map $\pi_i\colon S\to F(i)$ preserves suprema of increasing sequences.
\end{lma}
\begin{proof}
We may assume that $S$ is defined as in \eqref{pgr:prodinvlim:eq2}.
To verify \axiomO{1} for $S$, let $(s^{(n)})_{n\in\NN}$ be an increasing sequence in $S$.
For each $n$, we have $s^{(n)}=(s_i^{(n)})_{i\in I}$.
Since the order in $S$ is defined componentwise and $F(i)$ satisfies \axiomO{1} for each $i$, we may set $s_i:=\sup_n s_i^{(n)}$.
To verify that $s:=(s_i)_i$ belongs to $S$, let $f\colon i\to j$ be a morphism in $I$.
Since $F(f)\colon F(i)\to F(j)$ preserves suprema of increasing sequences, we deduce that
\[
F(f)(s_i)
= F(f)(\sup_n s_i^{(n)})
= \sup_n F(f)(s_i^{(n)})
= \sup_n s_j^{(n)}
= s_j.
\]
Thus, $s$ belongs to $S$, and it is easy to verify that $s$ is the supremum of $(s^{(n)})_n$ in~$S$.
This shows that $S$ satisfies \axiomO{1}.

The first part of the proof shows immediately that the projection maps $\pi_i$ preserve suprema of increasing sequences.
Moreover, since addition in $S$ is also taken componentwise, we obtain that $S$ satisfies \axiomO{4}.
\end{proof}

%==========================================================================================
\begin{dfn}
\label{dfn:auxPomLimit}
Let $I$ be a small category, and let $F\colon I\to\CatQ$ be a functor.
Let $(S,\pi)$ be the limit of $F$ in $\CatPom$ with $S$ defined as in \eqref{pgr:prodinvlim:eq2}.
Define a binary relation $\prec_{\mathrm{pw}}$ on $S$ by setting $s\prec_{\mathrm{pw}} t$ for $s=(s_i)_{i\in I}$ and $t=(t_i)_{i\in I}$ in $S$ if and only if $s_i\prec t_i$ in $F(i)$ for each $i\in I$. 
\end{dfn}

%==========================================================================================
\begin{prp}
\label{prp:Qcomplete}
The category $\CatQ$ is complete.
More precisely, given a small category $I$, and a functor $F\colon I\to\CatQ$, let $(S,(\pi_i)_i)$ be the limit of $F$ in $\CatPom$, with $S$ defined as in \eqref{pgr:prodinvlim:eq2}.
%Let $\prec_{\mathrm{pw}}$ be the relation on $S$ from \autoref{dfn:auxPomLimit}.
Then $(S,\prec_{\mathrm{pw}})$ is a $\CatQ$-semigroup and $\pi_i\colon S\to F(i)$ is a $\CatQ$-morphism for each $i\in I$.
Moreover, $(S,\prec_{\mathrm{pw}})$ with $(\pi_i)_i$ is a limit of $F$ in $\CatQ$.
\end{prp}
\begin{proof}
By \autoref{prp:limitPomO1O4}, $S$ satisfies \axiomO{1} and \axiomO{4}.
It is straightforward to check that $\prec_{\mathrm{pw}}$ defines an additive auxiliary relation on $S$.
Thus, $(S,\prec_{\mathrm{pw}})$ is a $\CatQ$-semigroup.
Given $i\in I$, the projection map $\pi_i\colon S\to F(i)$ clearly preserves the auxiliary relation.
By \autoref{prp:limitPomO1O4}, $\pi_i$ is also a $\CatPom$-morphism that preserves suprema of increasing sequences, and consequently $\pi_i$ is a $\CatQ$-morphism.
To simplify, we denote the $\CatQ$-semigroup $(S,\prec_{\mathrm{pw}})$ by $S$.
It follows that $(S,\pi)$ is a cone to $F$.

To show that $(S,\pi)$ is a limit of $F$ in $\CatQ$, let $(T,\sigma)$ be a cone to $F$.
By the universal property of the limit in $\CatPom$, there is a (unique) $\CatPom$-morphism $\sigma\colon T\to S$ such that $\sigma_i=\pi_i\circ\sigma$ for all $i\in I$.
Note that, $\sigma(t)=(\sigma_i(t))_{i\in I}$, for $t\in T$.
We have to show that $\sigma$ is a $\CatQ$-morphism.

To verify that $\sigma$ preserves the auxiliary relation, let $t',t\in T$ satisfy $t'\prec t$.
Since $\sigma_i$ preserves the auxiliary relation, we obtain that $\sigma_i(t')\prec\sigma_i (t)$, for each $i\in I$.
By definition, we then have $\sigma(t')\prec_{\mathrm{pw}}\sigma(t)$ in $S$, as desired.

To show that $\sigma$ preserves suprema of increasing sequences, let $(t_n)_n$ be an increasing sequence in $T$.
Set $t=\sup_n t_n$.
Since $\sigma_i$ preserves suprema of increasing sequences, we have $\sigma_i(t)=\sup_n\sigma_i(t_n)$ for each $i\in I$.
We deduce that
\[
\sigma(t)
= (\sigma_i(t))_i
= (\sup_n \sigma_i(t_n))_i
= \sup_n (\sigma_i(t_n))_i
= \sup_n \sigma(t_n). \qedhere
\]
\end{proof}

%===================================================
\begin{rmk}
It follows from \autoref{prp:Qcomplete} that $\CatQ$ has products.
The product in $\CatQ$ of a family $(S_i)_{i\in I}$ of $\CatQ$-semigroups is  precisely $\CatQProd_{i}S_i = \big( \CatPomProd_{i}S_i, \prec_{\mathrm{pw}} \big)$
%, the product of the underlying \pom{s} equipped with the coordinatewise auxiliary relation. 
\end{rmk}

%==========================================================================================
\begin{thm}
\label{prp:CuComplete}
The category $\CatCu$ is complete.
More precisely, given a small category~$I$ and a functor $F\colon I\to\CatCu$, let $(S,(\pi_i)_i)$ be the limit of $F$ in $\CatQ$.
Let $\tau(S)$ be the $\tau$-completion of $S$.
Setting $S_i:=\tau(F(i))$ and identifying it with $F(i)$, we set $\psi_i:=\tau(\pi_i)\colon\tau(S)\to\tau(S_i)\cong S_i$.
Then $\tau(S)$ together with $(\psi_i)_i$ is the limit of $F$ in $\CatCu$, that is:
\[
\CatCuinvLim F 
= \tau\left( \CatQinvLim F \right)
= \tau\left( \CatPominvLim F, \llpw \right)
\]
\end{thm}
\begin{proof}
By \autoref{prp:CuCoreflQ}, $\CatCu$ is a coreflective subcategory of $\CatQ$, which in turn is complete by \autoref{prp:Qcomplete}.
It follows that $\CatCu$ is complete, since completeness passes to coreflective subcategories;
see, for example, \cite[Proposition~3.5.3]{Bor94HandbookCat1}.
\end{proof}

%==========================================================================================
\begin{cor}
\label{cor:Cuproductinverselimitspullbacks}
The category $\CatCu$ has arbitrary products, arbitrary inverse limits, and finite pullbacks.
In particular:
\begin{enumerate}
\item
Let $(S_i)_{i\in I}$ be a family of $\CatCu$-semigroups.
Then
\[
\CatCuProd\limits_{i\in I} S_i
=\tau\left( \CatQProd\limits_{i\in I} S_i \right)
=\tau\left( \CatPomProd\limits_{i\in I} S_i, \llpw \right).
\]
\item
Let $((S_i)_{i\in I}, (\varphi_{ij})_{i,j\in I, i\geq j})$ be an inverse system of $\CatCu$-semigroups.
Then
\[
\CatCuinvLim_{i\in I} S_i = \tau\left( \CatQinvLim\limits_{i\in I} S_i \right)\,.
\]
\end{enumerate}
\end{cor}

%===================================================
The proof of the following result is straightforward, using that the $\tau$-construction applied to $\CatCu$-semigroups is vacuous;
see \cite[Proposition~4.10]{AntPerThi17arX:AbsBivariantCu}.

%===================================================
\begin{prp}
\label{prp:FinProdInCu}
Let $(S_j)_{j\in J}$ be a \emph{finite} family of \CuSgp{s}.
Then $\CatCuProd_{j\in J}S_j$ is naturally isomorphic to $\CatPomProd_j S_j$, the (set-theoretic) product of the $S_j$ equipped with componentwise order and addition.
\end{prp}

%==========================================================================================
%==========================================================================================
\subsection{Cocompleteness} 

%==========================================================================================
In this subsection we prove that $\CatCu$ is a cocomplete category;
see \autoref{prp:CuCocomplete}.
In particular, it has coproducts and pushouts.

%==========================================================================================
\begin{pgr}
\label{pgr:WfinCoproduct}
Let $\CatC$ be a category that has finite coproducts and inductive limits. Let $(X_j)_{j\in J}$ be a family of objects in $\CatC$. Denote $F\Subset J$ to mean that $F$ is a finite subset of $J$. The collection of such subsets is upward directed with respect to inclusion, and we have 
\[
\coprod_{j\in J}X_j\cong\varinjlim\limits_{F\Subset J}\coprod_{j\in F} X_j.
\]
Thus, $\CatC$ has coproducts; see, for example, \cite[Theorem~IX.1.1, p.212]{MacLan71Categories}.

Now assume that $\CatC$ also has finite products and a zero object.
Using universality we obtain, for $F\Subset J$, a natural morphism $\coprod_{j\in F}X_j\to\prod_{j\in F}X_j$.
Given $F\subseteq G\Subset J$, the existence of a zero object allows us to construct a canonical morphism $\prod_{j\in F}X_j\to\prod_{j\in G}X_j$.
In this way, one obtains an inductive system.
Its inductive limit is termed the \emph{direct sum} of the family $(X_j)_{j\in J}$ and is denoted by $\bigoplus_{j\in J}X_j$.
It follows from our discussion that there is a natural map
\[
\coprod_{j\in J}X_j \cong\varinjlim\limits_{F\Subset J}\coprod_{j\in F} X_j 
\rightarrow \varinjlim_{F\Subset J}\prod_{j\in F}X_j=\bigoplus_{j\in J} X_j.
\]
In the particular case that finite products and finite coproducts are isomorphic, it follows that this morphism is in fact an isomorphism and hence we may identify the coproduct of a family of objects in the category with their direct sum.
Therefore, it should not be confusing that we use the symbol `$\bigoplus$' to denote both.
As we shall see below, this is the case for the category $\CatCu$. 

Next, suppose that $\CatC$ has arbitrary products.
Since there are maps $\bigoplus_{j\in J}X_j\to X_i$ for each $i\in J$, we obtain by the universal property of the product a natural map $\bigoplus_{j\in J}X_j\to\prod_{j\in J}X_j$ and a commutative diagram
\[
\xymatrix@R-10pt{
\coprod_{j\in J}X_j \ar[r] \ar[dr]+<-2.3em,1.2ex>
& \prod_{j\in J}X_j \\
& \bigoplus_{j\in J}X_j. \ar[u]
}
\]

Now, let $(S_j)_{j\in J}$ be a finite family of $\CatW$-semigroups.
Let $\bigoplus_{j}S_j$, together with maps $\psi_i\colon S_i\to \bigoplus_{j\in J}S_j$ for $i\in J$, be the coproduct of the underlying abelian monoids.
Note that $\bigoplus_{j}S_j$ is a $\CatW$-semigroup when endowed with the componentwise auxiliary relation.
Moreover, this turns each $\psi_i$ into a $\CatW$-morphism.
It follows easily that this is a coproduct of $(S_j)_{j\in J}$ in the category $\CatW$. It is straightforward to verify that it also agrees with the finite product of $(S_j)_{j\in J}$ in $\CatW$. Thus, according to the above discussion, we shall denote this coproduct by $\CatWSum_{j}S_j$.
\end{pgr}

%==========================================================================================
\begin{lma}
\label{prp:Wcoproduct}
The category $\CatW$ has coproducts.
\end{lma}
\begin{proof}
By \autoref{pgr:WfinCoproduct}, $\CatW$ has finite coproducts.
An argument analogous to \cite[Theorem~2.1.10, p.15]{AntPerThi18TensorProdCu} shows that $\CatW$ has inductive limits.
This implies that $\CatW$ has coproducts, as discussed in \autoref{pgr:WfinCoproduct}.
\end{proof}

%==========================================================================================
\begin{lma}
\label{lma:prepforcoeqs}
Let $S$, $T$ be $\CatW$-semigroups, and let $\varphi,\psi\colon S\to T$ be $\CatW$-morphisms.
For $x,y\in T$, set $x\sim_0 y$ provided there are $z\in T$, and $a,b\in S$ such that
\[
x=z+\varphi(a)+\psi(b) \andSep y=z+\varphi(b)+\psi(a).
\]
Let $\sim$ be the transitive closure of the relation $\sim_0$.
Then $\sim$ is a congruence on $T$.

For $x,y\in T$, set $x\precsim_0 y$ provided there are $a,b\in T$ such that
\[
x \sim a \prec b \sim y.
\]
Let $\precsim$ be the transitive closure of $\precsim_0$.
Then $\precsim$ is a transitive, additive relation.

Given $u,x,y\in T$ satisfying $u\precsim x\sim y$ there exists $y'\in T$ such that $u\precsim y'\prec y$.
\end{lma}
\begin{proof}
It is straightforward to check that $\sim_0$ is additive.
It follows that $\sim$ is a congruence, and that $\precsim$ is additive and transitive.

Claim~1:
Let $u,x,y\in T$ satisfy $u\prec x\sim_0 y$.
Then there exists $y'\in T$ such that $u\precsim y'\prec y$.
To prove the claim, choose $z\in T$ and $a,b\in S$ such that
\[
x = z+\varphi(a)+\psi(b), \andSep y = z+\varphi(b)+\psi(a).
\]
Apply \axiomW{4} to choose $z',z'',v,w\in T$ such that
\[
z''\prec z'\prec z, \andSep  v\prec\varphi(a), \andSep w\prec\psi(b), \andSep u\prec z''+v+w.
\]
Now, as $\varphi$ and $\psi$ are continuous, there are $a', b'\in S$ such that $a'\prec a$, $b'\prec b$,  $v\prec\varphi(a')$, and $w\prec\psi(b')$.
Set $x':=z'+\varphi(b')+\psi(c')$, and $y':=z'+\varphi(c')+\psi(b')$.
Then
\[
u \prec z''+v+w \prec z'+\varphi(a')+\psi(b') = x'.
\]
We have $x'\sim_0 y'$, and therefore $u\precsim y'$.
Further, we have $y'\prec y$ since both $\varphi$ and $\psi$ preserve the auxiliary relation.

Claim~2:
Let $u,x,y\in T$ satisfy $u\prec x\sim y$.
Then there exists $y'\in T$ such that $u\precsim y'\prec y$.
To prove the claim, we use that $\sim$ is the transitive closure of $\sim_0$ to choose $n\in\NN$ and elements $x_0,\dots,x_n\in T$ such that
\[
x=x_0 \sim_0 x_1 \sim_0 \ldots \sim_0 x_n=y.
\]
Applying claim~1 to $u\prec x_0\sim_0 x_1$, we obtain $x_1'\in T$ with $x_0\precsim x_1'\prec x_1$.
Applying claim~1 repeatedly, we obtain $x_k'$ for $k=1,\ldots,n$ with $x_{k-1}'\precsim x_k'\prec x_k$.
Then $u\precsim x_n'\prec x_n$, which shows that $y':=x_n'$ has the desired properties.

Claim~3:
Let $u,x,y\in T$ satisfy $u\precsim_0 x\sim y$.
Then there exists $y'\in T$ such that $u\precsim y'\prec y$.
To prove the claim, choose $a,b\in T$ with $u\sim a\prec b\sim x$.
Then $a\prec b\sim y$.
Applying claim~2, we obtain $y'$ such that $a\precsim y'\prec y$.
It follows that $u\precsim y'\prec y$, as desired.

Finally, to prove the last statement of the lemma, let $u,x,y\in T$ satisfy $u\precsim x\sim y$.
Since $\precsim$ is the transitive closure of $\precsim_0$, we can choose $x'\in T$ such that
\[
u \precsim x' \precsim_0 x \sim y.
\]
By claim~3, we obtain $y'\in T$ with $x'\precsim y'\prec y$.
Then $u\precsim y'\prec y$, as desired.
\end{proof}

%==========================================================================================
\begin{lma}
\label{prp:Wcoeqs}
The category $\CatW$ has coequalizers.
\end{lma}
\begin{proof}
Let $S$ and $T$ be $\CatW$-semigroups, and let $\varphi,\psi\colon S\to T$ be $\CatW$-morphisms.
For ease of notation, we shall denote the auxiliary relations of both $S$ and $T$ by $\prec$.
We have to construct a $\CatW$-semigroup $L$ and a universal $\CatW$-morphism $\eta\colon T\to L$ such that $\eta\circ\varphi=\eta\circ\psi$.

Let $\sim$ be the congruence on $T$ defined as in \autoref{lma:prepforcoeqs}.
Notice that, with this definition, we have $\varphi(d)\sim\psi(d)$ for any $d\in S$.
Set $L:=T/\sim$, and let $\eta\colon T\to L$ denote the quotient map.
Given $x\in T$, we let $[x]$ denote the class of $x$ in $L$.
Thus, we have $\eta(x)=[x]$, for $x\in T$.

Let $\precsim$ be the transitive relation on $T$ defined as in \autoref{lma:prepforcoeqs}.
This induces a binary relation $\prec$ on $L$ by setting $[x]\prec [y]$ if and only if $x\precsim y$, for $x,y\in T$.
It follows from \autoref{lma:prepforcoeqs} that $\prec$ is a well-defined, transitive relation on $L$.

We claim that $(L,\prec)$ is a $\CatW$-semigroup.
In order to establish this claim, we only need to verify axioms \axiomW{1} and \axiomW{4}, as \axiomW{3} is trivially satisfied.

To verify \axiomW{1}, let $x\in T$.
Using that $T$ satisfies \axiomW{1}, we can choose a $\prec$-increasing sequence $(x_n)_n$ that is cofinal in $x^\prec:=\{x'\in T : x'\prec x\}$.
It is clear that $x_n\prec x_{n+1}$ implies $[x_n]\prec[x_{n+1}]$.
Thus, $([x_n])_n$ is a $\prec$-increasing sequence in $[x]^\prec$.
To show that this sequence is cofinal in $[x]^\prec$, let $a\in T$ satisfy $[a]\prec[x]$.
By definition, we have $a\precsim x$.
By \autoref{lma:prepforcoeqs}, there exists $x'\in T$ with $a\precsim x'\prec x$.
Choose $n$ such that $x'\prec x_n$.
Then $[a]\prec[x_n]$, as desired.

To verify \axiomW{4}, let $a,x,y\in T$ satisfy $[a]\prec[x]+[y]=[x+y]$.
By definition, we have $a\precsim x+y$.
By \autoref{lma:prepforcoeqs}, there exists $u\in T$ with $a\precsim u\prec x+y$.
Using that $T$ satisfies \axiomW{4}, we can choose $x',y'\in T$ such that $u\prec x'+y'$, $x'\prec x$, and $y'\prec y$.
Then the element $[u]$ in $L$ satisfies $[a]\prec[u]$, $[u]\prec[x']+[y']$, $[x']\prec[x]$, and $[y']\prec[y]$, as desired.

We have verified that $L$ is a $\CatW$-semigroup.
Moreover, the map $\eta\colon T\to L$ is a monoid morphism and preserves the auxiliary relation.
To verify that $\eta$ is continuous, let $a,b\in T$ satisfy $[a]\prec\eta(b)$.
Then $a\precsim b$.
By \autoref{lma:prepforcoeqs}, we can choose $b'\in T$ such that $a\precsim b'\prec b$.
Then $b'$ has the desired properties.
It is clear by the construction that $\eta\circ\varphi=\eta\circ\psi$.

We now claim that $(L,\eta)$ is a coequalizer of the pair $(\varphi,\psi)$.
Thus, suppose that $R$ is a $\CatW$-semigroup and $h\colon T\to R$ is a $\CatW$-morphism such that $h\circ\varphi=h\circ\psi$.

Define $\tilde{h}\colon L\to R$ by $\tilde{h}([x]):=h(x)$, for $x\in T$.
If we prove that $\tilde{h}$ is $\CatW$-morphism, then clearly it will be be the unique $\CatW$-morphism satisfying $\tilde{h}\circ\eta=h$.

In order to see that $\tilde{h}$ is well defined, let $x,y\in T$ satisfy $[x]=[y]$, that is, $x\sim y$.
We need to show that $h(x)=h(y)$.
Since $\sim$ is the transitive closure of $\sim_0$, we may assume that $x\sim_0 y$.
By definition, we can choose $a,b\in S$ and $z\in T$ such that
\[
x=z+\varphi(a)+\psi(b) \andSep y=z+\varphi(b)+\psi(a)\,.
\]
Using at the second step that $h\circ \varphi=h\circ\psi$, we obtain
\[
h(x)
= h(z) + h(\varphi(a)) + h(\psi(b))
= h(z) + h(\psi(a)) + h(\varphi(b))
= h(y).
\]

Using that $h$ is additive, it follows that $\tilde{h}$ is additive.
To show that $\tilde{h}$ preserves the auxiliary relation, let $x,y\in T$ satisfy $[x]\prec [y]$.
Then, $x\precsim y$.
Since $\precsim$ is the transitive closure of $\precsim_0$, we may assume that $x\precsim_0 y$.
By definition, we can choose $a,b\in T$ such that
\[
x\sim a \prec b \sim y.
\]
Using at the third step that $h$ preserves the auxiliary relation, we deduce
\[
\tilde{h}([x]) = h(x) = h(a) \prec h(b) = h(y) = \tilde{h}([y]).
\]

Finally, to show that $\tilde{h}$ is continuous, let $x\in T$ and $z\in R$ satisfy $z\prec \tilde{h}([x])$.
This means that $z\prec h(x)$, and since $h$ is continuous, there is $x'\in T$ with $x'\prec y$ and $z\leq h(x')$.
We have $[x']\prec [x]$ in $L$, and clearly $z\leq \tilde{h}([x'])$ in $R$, which shows that $[x']$ has the desired properties.
\end{proof}

%==========================================================================================
\begin{prp}
\label{prp:WCocomplete}
The category $\CatW$ is cocomplete.
\end{prp}
\begin{proof}
By Lemmas~\ref{prp:Wcoproduct} and \ref{prp:Wcoeqs}, $\CatW$ has small coproducts and coequalizers.
This implies that $\CatW$ is cocomplete;
% since in general a category is cocomplete if and only if it has small coproducts and coequalizers;
see for example \cite[Corollary~V.2.2, p.113]{MacLan71Categories} or (the dual statement of) \cite[Theorem~2.8.1]{Bor94HandbookCat1}.
\end{proof}

%==========================================================================================
\begin{thm}
\label{prp:CuCocomplete}
The category $\CatCu$ is cocomplete.
More precisely, if $I$ is a small category and $F\colon I\to \CatCu$ is a functor, then $\CatCuLim F=\gamma(\CatWLim F)$.
\end{thm}
\begin{proof}
By \autoref{prp:WCocomplete}, $\CatW$ is cocomplete.
By \autoref{prp:CuReflW}, $\CatCu$ is a full, reflective subcategory of $\CatW$, with reflector given by $\gamma\colon\CatW\to\CatCu$.
It follows that $\CatCu$ is cocomplete, with the colimit of a functor $F$ given by the reflection of the colimit of $F$ in $\CatW$;
see \cite[Proposition~3.5.3]{Bor94HandbookCat1}.
\end{proof}

%==========================================================================================
\begin{cor}
\label{cor:Cucoproductpushouts}
The category $\CatCu$ has coproducts and pushouts.
In particular, the coproduct of a family $(S_j)_{j\in J}$ of \CuSgp{s} is $\CatCuCoprod_{j} S_j = \gamma \big( \CatWSum_{j} S_j \big)$.
\end{cor}

%===================================================
The following result will be used in \autoref{sec:dirsum}.
Its proof, as the one of \autoref{prp:FinProdInCu}, is also straightforward, using that the $\gamma$-completion applied to $\CatCu$-semigroups is vacuous.

%===================================================
\begin{prp}
\label{prp:FinCoprodInCu}
Let $(S_j)_{j\in J}$ be a \emph{finite} family of \CuSgp{s}.
Then $\CatCuCoprod_{j}S_j$ is naturally isomorphic to $\CatCuProd_{j\in J}S_j$.
\end{prp}

%==========================================================================================
\begin{pgr}
\label{pgr:absbivar}
It was shown in \cite{AntPerThi17arX:AbsBivariantCu} that $\CatCu$ is a closed, symmetric, monoidal category.
The tensor product construction in the category $\CatCu$ was introduced and developed in \cite[Chapter~6]{AntPerThi18TensorProdCu}.
It was proved in \cite[Proposition~6.4.1]{AntPerThi18TensorProdCu} that the bifunctor $\freeVar\otimes\freeVar\colon\CatCu\times\CatCu\to\CatCu$ is continuous in both variables.
This result is recovered below.
On the other hand, \cite[Theorem~5.11]{AntPerThi17arX:AbsBivariantCu} shows that, for any $\CatCu$-semigroup $T$, the functor $\freeVar\otimes_\CatCu T$ has a right adjoint, denoted by $\ihom{T,\freeVar}$.
The bifunctor $\ihom{\freeVar,\freeVar}\colon\CatCu\times\CatCu\to\CatCu$ is referred to as the \emph{internal-hom bifunctor}.
\end{pgr}

%==========================================================================================
\begin{thm}
\label{thm:inthominvlim}
Let $T$ be a $\CatCu$-semigroup, let $((S_i)_{i\in I}, (\varphi_{j,i})_{i\leq j})$ be an inductive system in $\CatCu$, and let $((R_i)_{i\in I}, (\psi_{j,i})_{i\leq j})$ be an inverse system in $\CatCu$.
Then:
\begin{enumerate}
\item
We have an inductive system $((S_i\otimes T)_{i\in I}, (\varphi_{j,i}\otimes\mathrm{id}_T)_{i\leq j})$ in $\CatCu$, and
\[
(\CatCuLim\limits_{i\in I} S_i)\otimes T\cong\CatCuLim\limits_{i\in I} (S_i\otimes T).
\]
\item
We have an inverse system $((\ihom{T,R_i})_{i\in I}, ((\psi_{j,i})_*)_{i\leq j})$ in $\CatCu$, and
\[
\ihom{T,\CatCuinvLim\limits_{i\in I} R_i}\cong\CatCuinvLim\limits_{i\in I}\ihom{T, R_i}.
\]
\item
We have an inverse system $((\ihom{S_i,T})_{i\in I}, ((\varphi_{j,i}^*)_{i\leq j})$ in $\CatCu$, and
\[
\ihom{\CatCuLim\limits_{i\in I}S_i, T}\cong\CatCuinvLim\limits_{i\in I}\ihom{S_i,T}.
\]
\end{enumerate}
\end{thm}
\begin{proof}
The functors $F_T:=\freeVar\otimes T$ and $G_T:=\ihom{T,\freeVar}$ form an adjoint pair by \cite[Theorem~5.10]{AntPerThi17arX:AbsBivariantCu}. Then (1) and (2) follow from general category theory.

Let us prove statement~(3). Using the notation in \cite[Paragraph~5.5]{AntPerThi17arX:AbsBivariantCu}, we have
that $((\ihom{S_i,T})_{i\in I}, (\varphi_{i,j}^*)_{i,j\in I})$ is an inverse system of $\CatCu$-semigroups.
Let $R$ be any other $\CatCu$-semigroup. Since $\tau\colon\CatQ\to\CatCu$ is a coreflection (\autoref{prp:CuCoreflQ}), we obtain natural bijections:
\[
\CatCuMor\big( R,\CatCuinvLim\ihom{S_i,T} \big)
\cong \CatCuMor\big( R,\tau(\CatQinvLim \ihom{S_i,T}) \big)
\cong \CatQMor\big( R,\CatQinvLim \ihom{S_i,T} \big).
\]

Given an inverse system of $\CatQ$-semigroups $(Y_i)_i$, the set $\CatQMor(R,\CatQinvLim_i Y_i)$ is naturally bijective to the (set-theoretic) inverse limit $\varprojlim_i\CatQMor(R, Y_i)$.
Similarly, given an inductive system of $\CatW$-semigroups $(X_i)_i$, the set $\CatWMor(\CatWLim X_i, T)$ is naturally bijective to $\varprojlim\CatWMor(X_i,T)$. Using these observations, and the adjunction between tensor product and internal-hom, the result follows.
\end{proof}

%==========================================================================================
%==========================================================================================
\section{Scaled Cuntz semigroups}
\label{sec:scaled}

%==========================================================================================
In this section, we introduce the category $\CatCuScaled$ of scaled \CuSgp{s} (see \cite{PerTomWhiWin14CuStabilityCloseCa}), which we show to be complete and cocomplete.

%==========================================================================================
\begin{dfn}
Let $S$ be a \CuSgp.
A \emph{scale} for $S$ is a downward hereditary subset $\Sigma\subseteq S$ that is closed under passing to suprema of increasing sequences in $S$, and that generates $S$ as an ideal.
We call $(S,\Sigma)$ a \emph{scaled \CuSgp}.

Given scaled \CuSgp{s} $(S,\Sigma)$ and $(T,\Theta)$, a \emph{scaled \CuMor} is a \CuMor{} $\alpha\colon S\to T$ satisfying $\alpha(\Sigma)\subseteq\Theta$.
We let $\CatCuScaled$ denote the category of scaled \CuSgp{s} and scaled \CuMor{s}.
\end{dfn}

%==========================================================================================
\begin{pgr}
\label{pgr:CuScaled}
Let $A$ be a \ca{}.
Then
\[
%\Sigma_A := \big\{ x\in\Cu(A) : \forall x'\ll x \exists a\in A_+ : x\leq[a] \text{ for some } a\in A_+ \big\}
\Sigma_A := \big\{ x\in\Cu(A) : \text{for every } x'\ll x \text{ there exists } a\in A_+ \text{ such that } x'\ll [a] \big\}
\]
is a scale for $\Cu(A)$.
We call $\CatCuScaled(A):=(\Cu(A),\Sigma_A)$ the \emph{scaled Cuntz semigroup} of $A$.
Given a \stHom{} $\varphi\colon A\to B$ between \ca{s}, the induced \CuMor{} $\Cu(\varphi)\colon\Cu(A)\to\Cu(B)$ (see \autoref{pgr:CuA}) maps $\Sigma_A$ into $\Sigma_B$.
Hence, we obtain a functor $\CatCa\to\CatCuScaled$.
In Theorems~\ref{prp:CuScaledIndLim}, \ref{thm:CuProd} and~\ref{thm:CuUltraprod}, we show that this functor preserves inductive limits and (ultra)products.
%$\CatCuScaled(\freeVar)\colon
\end{pgr}

%==========================================================================================
The following result is straightforward to prove.
We omit the details.

%==========================================================================================
\begin{prp}
\label{pgr:scalesIdealsQuotients}
Let $A$ be a \ca, let $J\subseteq A$ be an ideal, and let $\pi\colon A\to A/J$ denote the quotient map.
Then
\[
\Sigma_J = \Cu(J)\cap\Sigma_A,\andSep
\Sigma_{A/J} = \Cu(\pi)(\Sigma_A).
\]
\end{prp}

%==========================================================================================
\begin{pgr}
Given a \CuSgp{} $S$, the whole of $S$ is a scale for $S$, and we say that $(S,S)$ is \emph{trivially scaled}.
A map between \CuSgp{s} is a \CuMor{} if and only if it is a scaled \CuMor{} when both \CuSgp{s} are equipped with trivial scales.
Hence, $\CatCu$ is isomorphic to the full subcategory of trivially scaled \CuSgp{s} in $\CatCuScaled$.
In the other direction, we have the forgetful functor $\forget\colon\CatCuScaled\to\CatCu$, that maps $(S,\Sigma)$ to $S$.
For every scaled \CuSgp{} $(S,\Sigma)$ and every \CuSgp{} $T$, there is a natural bijection
\[
\CatCuMor\big( S, T \big) \cong \CatCuScaledMor\big( (S,\Sigma), (T,T) \big),
\]
showing that $\forget$ is a left adjoint to the inclusion $\CatCu\to\CatCuScaled$, and thus $\CatCu$ is a full, reflective subcategory of $\CatCuScaled$.
\end{pgr}

%==========================================================================================
\begin{pgr}
Let $I$ be a small category, and let $F\colon I\to\CatCuScaled$, $i\mapsto F(i)=(S_i,\Sigma_i)$ be a functor.
Consider the underlying functor $I\to\CatPom$, $i\mapsto S_i$, and let $(S,(\pi_i)_i)$ be the limit of $F$ in $\CatPom$ (see \autoref{pgr:prodinvlim}) with $S$ defined as in \eqref{pgr:prodinvlim:eq2}.
Set $\Sigma_0 := S \cap \prod_i \Sigma_i$.
Note that
\[
\Sigma_0
= \left\{ (s_i)_{i\in I}\in\prod_{i\in I} \Sigma_i :  F(f)(s_i)=s_j \text{ for all } f\colon i\to j \text{ in }I \right\}.
\]
Then $\Sigma_0$ is a downward hereditary subset of $S$ satisfying $\pi_i(\Sigma_0)\subseteq \Sigma_i$ for each $i$.

The limit of the underlying functor $\forget\circ F\colon I\to\CatCu$, $i\mapsto S_i$, is given as
$\tau( S, \llpw )$ together with maps $\psi_i:=\tau(\pi_i)\colon\tau(S,\llpw)\to\tau(S_i,\ll)\cong S_i$;
see \autoref{prp:CuComplete}.
Set
\[
\Sigma := \big\{ [(\vect{x}_t)_{t\leq 0}] \in \tau( S, \llpw ) : \vect{x}_t \in \Sigma_0 \text{ for all } t<0 \big\}.
\]
Then $\Sigma$ is a downward hereditary subset of $\tau( S, \llpw )$ that is closed under passing to suprema of increasing sequences.
Let $\langle \Sigma \rangle$ denote the ideal of $\tau( S, \llpw )$ generated by $\Sigma$.
Then $(\langle\Sigma\rangle,\Sigma)$ is a scaled \CuSgp{}.
Moreover, for each $i\in I$ we have $\psi_i(\Sigma)\subseteq\Sigma_i$, which shows that $\psi_i\colon(\langle\Sigma\rangle,\Sigma)\to(S_i,\Sigma_i)$ is a scaled \CuMor.
One can show that this defines a limit for $F$ in $\CatCuScaled$.
We omit the details.
\end{pgr}

%==========================================================================================
\begin{thm}
\label{prp:CuScaledComplete}
The category $\CatCuScaled$ is bicomplete.
\end{thm}
\begin{proof}
It remains to show that $\CatCuScaled$ is cocomplete.
Let $I$ be a small category, and let $F\colon I\to\CatCuScaled$, $i\mapsto F(i)=(S_i,\Sigma_i)$ be a functor.
By \autoref{prp:CuCocomplete}, $\CatCu$ is cocomplete, and so there is a \CuSgp{} $S$ and \CuMor{s} $\sigma_i\colon S_i\to S$ that are compatible with $F(f)\colon S_i\to S_j$ for each $f\colon i\to j$ in $I$, and such that $(S,(\sigma_i)_i)$ is universal with these properties.
Let $\Sigma$ be the smallest subset of $S$ that contains $\bigcup_i\sigma_i(\Sigma_i)$, and that is downward hereditary and closed under suprema of increasing sequences in $S$, that is,
\[
\Sigma = \big\{ x\in S : \text{ for every } x'\ll x\ \text{ there exists } y\in \bigcup_{i\in I}\sigma_i(\Sigma_i) \text{ with } x'\leq y \big\}.
\]
Then $(S,\Sigma)$ is a scaled \CuSgp, and each $\sigma_i$ is a scaled \CuMor.
To show that this gives a limit for $F$ in $\CatCuScaled$, let $(T,\Theta)$ be a scaled \CuSgp{} together with compatible scaled \CuMor{s} $\tau_i\colon (S_i,\Sigma_i)\to(T,\Theta)$.
Using the universal property of $(S,(\sigma_i)_i)$, we obtain a (unique) \CuMor{} $\beta\colon S\to T$ such that $\tau_i=\beta\circ\sigma_i\colon S_i\to T$.
Since each $\tau_i$ is a scaled \CuMor{}, we have
\[
\beta\big( \bigcup_{i\in I}\sigma_i(\Sigma_i) \big)
= \bigcup_{i\in I}(\beta\circ\sigma_i)(\Sigma_i)
= \bigcup_{i\in I}\tau_i(\Sigma_i)
\subseteq \Theta,
\]
which implies that $\beta(\Sigma)\subseteq\Theta$, and thus $\beta$ is a scaled \CuMor.
%It follows that  with the desired properties.
\end{proof}

%==========================================================================================
\begin{thm}
\label{prp:CuScaledIndLim}
The scaled Cuntz semigroup functor preserves inductive limits:
Given an inductive system $((A_j)_{j\in J},(\varphi_{k,j})_{j\leq k})$ of \ca{s}, we have
\[
\CatCuScaled\big( \varinjlim_{j}A_j \big) \cong \varinjlim_{j} \CatCuScaled(A_j).
\]
\end{thm}
\begin{proof}
Set $A:=\varinjlim_{j}A_j$ and let $\varphi_{\infty,j}\colon A_j\to A$ denote the maps to the inductive limit.
They induce scaled \CuMor{s} $\psi_{\infty,j}\colon\CatCuScaled(A_j)\to\CatCuScaled(A)$, which in turn induce a scaled \CuMor{}
\[
\psi\colon \varinjlim_{j}\CatCuScaled(A_j) \to \CatCuScaled(A).
\]
It follows from \cite[Corollary~3.2.9, p.29]{AntPerThi18TensorProdCu} that $\psi$ is an order-isomorphism.
Let $\Sigma$ denote the scale  of $\varinjlim_{j}\CatCuScaled(A_j)$.
We have $\psi(\Sigma)\subseteq\Sigma_A$, and it remains to show that $\psi$ maps $\Sigma$ onto $\Sigma_A$.
Let $x\in\Sigma_A$, and let $x'\ll x$.
It is enough to verify $x'\in\psi(\Sigma)$.
%Choose $x'$ with $x''\ll x'\ll x$.
By definition of $\Sigma_A$, there exists $a\in A_+$ with $x'\ll[a]$.
Thus, there is $\varepsilon>0$ with $x'\ll[(a-\varepsilon)_+]$.
Choose $j$ and $b\in (A_j)_+$ such that $\|a-\varphi_{\infty,j}(b)\|<\varepsilon$.
Set $y:=[b]\in\Sigma_{A_j}$, and let $\tilde{y}$ be the image of $y$ in $\varinjlim_{j}\Cu(A_j)$.
Then $\tilde{y}\in\Sigma$, and 
\[
x' 
\ll [(a-\varepsilon)_+] 
\leq [\varphi_{\infty,j}(b)]
= \psi_{\infty,j}(y)
= \psi(\tilde{y}).
\]
It follows that $\psi^{-1}(x')\in\Sigma$, and thus $x'\in\psi(\Sigma)$.
\end{proof}

%==========================================================================================
%==========================================================================================
\section{Cuntz semigroups of products}
\label{sec:prod}

%==========================================================================================
In this section, we show that the scaled Cuntz semigroup functor preserves products;
see \autoref{thm:CuProd}.
%Recall that the product in $\CatCu$ is obtained by applying the $\tau$-construction to the product in $\CatQ$.
%Thus, given a family $(S_j)_{j\in J}$ of \CuSgp{s}, we have $\CatCuProd_{j}S_j = \tau(\CatPomProd_j S_j, \llpw )$.

%==========================================================================================
\begin{pgr}
\label{pgr:CuProductSetup}
Let $(A_j)_{j\in J}$ be a family of \ca{s}.
For each $i\in J$, the natural projection $\pi_i\colon\prod_j A_j\to A_i$ induces a \CuMor{} $\tilde{\pi}_i\colon\Cu(\prod_jA_j)\to \Cu(A_i)$. 
Then, by the universal property of the product, there is a unique \CuMor{}
\[
\Phi\colon \Cu(\prod_{j\in J}A_j)\to \CatCuProd_{j\in J} \Cu(A_j)
\]
such that $\tilde{\pi}_i=\sigma_i\circ\Phi$ for all $i\in I$, where $\sigma_i\colon\CatCuProd_j\Cu(A_j)\to\Cu(A_j)$ denotes the natural \CuMor{} associated to the product in the category $\CatCu$.

Let us describe the map $\Phi$ more concretely.
Given a positive element $\underline{a}=(a_j)_{j\in J}$ in $\prod_jA_j$, and $t\in(-\infty,0]$, set
\[
\varphi_t(\underline{a}) := ([(a_j+t)_+])_{j\in J}.
\]
Then $t\mapsto\varphi_t(\underline{a})$ is a path in $(\CatPomProd_j \Cu(A_j),\llpw)$, and we have
\[
\Phi([\underline{a}])=[(\varphi_t(\underline{a}))_{t\in(-\infty,0]}].
\]
\end{pgr}

%==========================================================================================
Next, we need to deal with the stabilization that occurs in the construction of $\Cu(\prod_j A_j)$. 
We note that even if each $A_j$ is stable, the product $\prod_j A_j$ is \emph{not} stable (unless $J$ is finite);
see \autoref{rmk:stableUltrapr}.
Nevertheless, we will see that $\prod_j A_j$ satisfies property~(S) of Hjelmborg-R{\o}rdam, which is enough for our purposes.

%==========================================================================================
\begin{dfn}[{Hjelmborg, R{\o}rdam, \cite{HjeRor98Stability}}]
A \ca{} $A$ is said to have property~(S) if for every $a\in A_+$ and every $\varepsilon>0$ there exist $b\in A_+$ and $x\in A$ such that $a=x^*x$, $b=xx^*$ and $\|ab\|\leq\varepsilon$.
\end{dfn}

%==========================================================================================
\begin{thm}[{\cite{HjeRor98Stability}}]
\label{prp:propS}
Let $A$ be a \ca{}.
If $A$ is stable, then it has property~(S).
The converse holds if $A$ is $\sigma$-unital (in particular, if $A$ is separable).
\end{thm}

%==========================================================================================
\begin{lma}
\label{prp:PropSProduct}
Property~(S) passes to products.
\end{lma}
\begin{proof}
Let $(A_j)_{j\in J}$ be a family of \ca{s} with property~(S).
To verify that $A:=\prod_{j\in J}A_j$ has property~(S), let $a=(a_j)_j\in A_+$ and $\varepsilon>0$. 
%Let $a_i\in A_i$, for each $i\in I$, such that $a=(a_i)_i$.
Given $j\in J$, use that $A_j$ satisfies~(S), to choose $b_j\in (A_j)_+$ and $x_j\in A_j$ such that
\[
a_j=x_j^*x_j, \andSep b_j=x_jx_j^*, \andSep \|a_jb_j\|\leq\varepsilon.
\]
Set $b:=(b_j)_j$ and $x:=(x_j)_j$.
For each $j$, we have $\|x_j\|=\|a_j\|^{1/2} \leq \|a\|^{1/2}$ and therefore $\|x\|=\sup_i\|x_i\|<\infty$.
Similarly, we have $\|b\|<\infty$.
Hence, $b$ and $x$ belong to $A$, and we have $a=x^*x$, $b=xx^*$ and $\|ab\|\leq\varepsilon$, as desired.
\end{proof}

%==========================================================================================
We remark that it is straightforward to verify that property~(S) passes to quotients, and hence to ultraproducts.
The following result is well-known (although we couldn't locate it in the literature).
It can be proved using Blackadar's technique of constructing separable sub-\ca{s} with prescribed properties (\cite[Section~II.8.5, p.176ff]{Bla06OpAlgs}), which is the \ca{ic} version of the downward L\"{o}wenheim-Skolem theorem from model theory.

%==========================================================================================
\begin{lma}
\label{prp:PropStoSepSubalg}
Let $A$ be a \ca{} with property~(S), and let $A_0\subseteq A$ be a separable sub-\ca{}.
Then there exists a separable sub-\ca{} $B\subseteq A$ with property~(S) (and hence $B$ is stable by \autoref{prp:propS}), and such that $A_0\subseteq B$.
\end{lma}

%==========================================================================================
\begin{lma}
\label{prp:unstableCuPropS}
Let $A$ be a \ca{} with property~(S), and let $p\in\KK$ be a rank one projection.
Then the map $\iota\colon A\to A\otimes\KK$, given by $\iota(a)=a\otimes p$, induces a natural bijection between Cuntz equivalence classes in $A_+$ and $(A\otimes\KK)_+$. 
Moreover, $a,b\in A_+$ satisfy $\iota(a)\precsim\iota(b)$ in $A\otimes\KK$ if and only if $a\precsim b$ in $A$.
\end{lma}
\begin{proof}
In general, if $D\subseteq B$ is a hereditary sub-\ca{}, and given $x,y\in D_+$, then $x$ is Cuntz (sub)equivalent to $y$ in $D$ if and only if $x$ is Cuntz (sub)equivalent to $y$ in $B$;
see e.g.\ \cite[Proposition~2.18]{Thi17:CuLectureNotes}.
Since $\iota$ identifies $A$ with a hereditary sub-\ca{} of $A\otimes\KK$, it follows that the map $\tilde{\iota}\colon (A_+)_{/\sim} \to\Cu(A)$, given by $\tilde{\iota}([a])=[a\otimes p]$, is a well-defined order-embedding.

To show that $\tilde{\iota}$ is surjective, let $a\in(A\otimes\KK)_+$.
Choose a separable sub-\ca{} $A_0\subseteq A$ such that $a$ belongs to $A_0\otimes\KK\subseteq A\otimes\KK$.
(For instance, let $(e_{kl})_{k,l\geq 1}$ denote matrix units for $\KK$ such that $p=e_{11}$, and let $A_0$ be generated by $\iota^{-1}(e_{1k}ae_{l1})$ for $k,l\geq 1$.)
Apply \autoref{prp:PropStoSepSubalg} to obtain a separable, stable sub-\ca{} $B\subseteq A$ with $A_0\subseteq B$.
Note that $a$ belongs to $B\otimes\KK$.
Since $B$ is stable, there exists a unitary in $M(B\otimes\KK)$ such that $uxu^*\in B\otimes p$ for each $x\in B\otimes\KK$.
Then, $uau^*$ belongs to $B\otimes p$, and hence to the image of $\iota$.
Further, $uau^*$ is Cuntz equivalent to $a$ in $M(B\otimes\KK)$, and consequently in the hereditary sub-\ca{} $B\otimes\KK$, and therefore also in $A\otimes\KK$. 
\end{proof}

%==========================================================================================
The following technical lemmas will be used to compute the Cuntz semigroups of product \ca{s}.

%==========================================================================================
\begin{lma}
\label{prp:PhiInjective}
Let $(A_j)_{j\in J}$ be a family of \ca{s}, and let $\underline{a}=(a_j)_{j\in J}$ and $\underline{b}=(b_j)_{j\in J}$ be positive elements in $\prod_jA_j$.
Then the following are equivalent:
\begin{enumerate}
\item
We have $\underline{a}\precsim\underline{b}$ in $\prod_jA_j$.
\item 
For every $\varepsilon>0$ there exists $\delta>0$ such that $(a_j-\varepsilon)_+\precsim(b_j-\delta)_+$ in $A_j$ for every $j\in J$.
\item
We have $\Phi([\underline{a}])\leq\Phi([\underline{b}])$ in $\CatCuProd_j\Cu(A_j)$.
\end{enumerate}
\end{lma}
\begin{proof}
As in \autoref{pgr:CuProductSetup}, set
\[
\varphi_t(\underline{a}) = ([(a_j+t)_+])_j,\andSep
\varphi_t(\underline{b}) = ([(b_j+t)_+])_j,
\]
for each $t\in(-\infty,0]$.
Then
\[
\Phi([\underline{a}]) = [ \varphi_t(\underline{a})_{t\in(-\infty,0]} ], \andSep
\Phi([\underline{b}]) = [ \varphi_t(\underline{b})_{t\in(-\infty,0]} ].
\]
By definition of the order in $\CatCuProd_j \Cu(A_j)=\tau(\CatPomProd_j\Cu(A_j),\llpw)$, we have $\Phi([\underline{a}])\leq\Phi([\underline{b}])$ if and only if for every $t<0$ there exists $t'<0$ such that $\varphi_t(\underline{a}) \llpw \varphi_{t'}(\underline{b})$.
Equivalently, for every $\varepsilon>0$ there exists $\delta>0$ such that $\varphi_{-\varepsilon}(\underline{a}) \llpw \varphi_{-\delta}(\underline{b})$.
We deduce that~(2) and~(3) are equivalent.

It follows directly from R{\o}rdam's Lemma (\autoref{pgr:CuA}) that~(1) implies~(2).
Conversely, assume~(2), and let us verify~(1) with R{\o}rdam's Lemma.
Let $\varepsilon>0$.
We need to find $\underline{s}$ in $\prod_j A_j$ such that $\underline{s}\underline{s}^*=(\underline{a}-\varepsilon)_+$ and $\underline{s}^*\underline{s}\in\Her(\underline{b})$.
%$verify that $(\underline{a}-\varepsilon)_+\precsim\underline{b}$ in $\prod_j A_j$.
By assumption, for each $j\in J$ there is $\delta>0$ such that $(a_j-\tfrac{\varepsilon}{2})_+\precsim (b_j-\delta)_+$ in $A_j$.
Apply R{\o}rdam's Lemma for $(a_j-\tfrac{\varepsilon}{2})_+\precsim (b_j-\delta)_+$ and $\tfrac{\varepsilon}{2}$ to obtain $s_j\in A_j$ such that
\[
(a_j-\varepsilon)_+ 
= \left( (a_j-\tfrac{\varepsilon}{2})_+  -\tfrac{\varepsilon}{2}\right)_+
= s_js_j^*, \andSep 
s_j^*s_j \in \Her((b_j-\delta)_+).
\]

Observe that $\| s_j \| \leq \| a_j \|^{1/2}$ for all $j\in J$, and hence $(\| s_j \|)_j$ is bounded.
Set $\underline{s}:=(s_j)_{j\in J}$.
Then $\underline{s}$ is an element in $\prod_j A_j$ satisfying $\underline{s}\underline{s}^*=(\underline{a}-\epsilon)_+$.
Let $f_\delta\colon\RR\to[0,1]$ be a continuous function satisfying $f(t)=0$ for $t\leq 0$ and $f(t)=1$ for $t\geq\delta$.
Then $f_\delta(b_j)s_j^*s_jf_\delta(b_j)=s_j^*s_j$ for each $j$, and therefore $f_\delta(\underline{b})\underline{s}^*\underline{s}f_\delta(\underline{b})=\underline{s}^*\underline{s}$.
This implies that $\underline{s}^*\underline{s}\in\Her(\underline{b})$, as desired.
\end{proof}

%==========================================================================================
\begin{lma}
\label{prp:PhiDense}
Let $A$ be a \ca{}, let $a\in A_+$, and let $x\in\Cu(A)$ with $x\ll[a]$.
Then there exists a positive, contractive $b\in A$ such that $x\ll[(b-\tfrac{1}{2})_+]$ and $[b]=[a]$.
\end{lma}
\begin{proof}
Choose $\delta>0$ such that $x\ll [(a-\delta)_+]$.
We may assume that $\delta<1/2$.
Let $h\colon\RR\to[0,1]$ be defined by
\begin{align*}
h(t)=
\begin{cases}
0, &\text{if } t\leq 0\\
\frac{1}{2\delta}t, &\text{if }0\leq t\leq 2\delta\\
1, &\text{if }2\delta\leq t
\end{cases}.
\end{align*}
Set $b:=h(a)$.
Then $[a]=[b]$.
Consider the functions $\RR\to\RR$, given by $t\mapsto (h(t)-\tfrac{1}{2})_+$ and $t\mapsto (t-\delta)_+$.
They have the same support (namely $(\delta,\infty)$), which implies that $(h(a)-1/2)_+$ and $(a-\delta)_+$ are Cuntz equivalent (first in the commutative \ca{} $C^*(a)\cong C(\sigma(a))$, and thus also in $A$).
Hence,
\[
x \ll [(a-\delta)_+] = [(h(a)-\frac{1}{2})_+] = [(b-\frac{1}{2})_+].\qedhere
\]
\end{proof}

%==========================================================================================
\begin{prp}
\label{prp:CuProdStable}
Let $(A_j)_{j\in J}$ be a family of stable \ca{s}.
Then, the map $\Phi\colon \Cu(\prod_{j}A_j)\to \CatCuProd_{j} \Cu(A_j)$ from \autoref{pgr:CuProductSetup} is an isomorphism.
\end{prp}
\begin{proof}
Set $A:=\prod_j A_j$.
By \autoref{prp:PropSProduct}, $A$ has property~(S).
Hence, by \autoref{prp:unstableCuPropS} there is a natural identification of $\Cu(A)$ with the Cuntz equivalence classes in $A_+$.
Thus, it remains to show that $\Phi$ is an order-isomorphism when restricted to Cuntz classes of elements in $A_+$.

It follows from \autoref{prp:PhiInjective} that $\Phi$ is an order-embedding.
It remains to show that $\Phi$ is surjective.
Since $\Phi$ is an order-embedding, and since $\Cu(A)$ has suprema of increasing sequences that moreover are preserved by $\Phi$, it is enough to verify that the image of $\Phi$ is order-dense.

Let $x,y\in \CatCuProd_j\Cu(A_j)$ satisfy $x\ll y$.
We will find $\underline{b}\in A_+$ such that $x\ll\Phi([\underline{b}])\ll y$.
Choose $\llpw$-increasing paths $(\vect{x}_t)_{t\in(-\infty,0]}$ and $(\vect{y}_t)_{t\in(-\infty,0]}$ in $\CatPomProd_j\Cu(A_j)$ representing $x$ and $y$, respectively.
It follows from \cite[Lemma~3.16]{AntPerThi17arX:AbsBivariantCu} that there exists $t<0$ such that $\vect{x}_0\llpw\vect{y}_t$.

Choose elements $x_{0,j}\in\Cu(A_j)$ such that $\vect{x}_0=(x_{0,j})_j$,
and choose positive elements $a_j\in A_j$ such that $\vect{y}_t = ([a_j])_j$.
Then, for each $j\in J$, we have $x_{0,j}\ll[a_j]$.
Applying \autoref{prp:PhiDense}, we obtain a positive, contractive $b_j\in A_j$ such that
\[
x_{0,j} \ll [(b_j-\frac{1}{2})_+],\andSep
[b_j]=[a_j].
\]
Set $\underline{b}:=(b_j)_j$, which is a positive, contractive element in $\prod_j A_j$.
We have
\[
\vect{x}_0
\llpw ([(b_j-\frac{1}{2})_+])_j
= \varphi_{-1/2}(\underline{b}),\andSep
\varphi_0(\underline{b})
= ([b_j])_j
= ([a_j])_j
= \vect{y}_t
\llpw \vect{y}_{t/2},
\]
and therefore
\[
x
=[(\vect{x}_t)_{t\leq 0}]
\ll [(\varphi_t(\underline{b}))_{t\leq 0}]
= \Phi([\underline{b}])
\ll [(\vect{y}_t)_{t\leq 0}] = y.
\]
This shows that $\underline{b}$ has the claimed properties.
\end{proof}

%==========================================================================================
\begin{rmk}
More generally, \autoref{prp:CuProdStable} (and consequently \autoref{prp:CuUltraprodStable}) holds for every family of \ca{s} that satisfy property~(S).
However, \autoref{exa:CuProd} shows that it does not hold for every family of \ca{s}.
\end{rmk}

%==========================================================================================
The proof of the following result is standard.
We therefore omit it.

%==========================================================================================
\begin{lma}
\label{prp:HerSubalg}
Let $A\subseteq B$ be a hereditary sub-\ca{}.
Then the inclusion $\iota\colon A\to B$ induces an order-embedding $\Cu(\iota)\colon\Cu(A)\to\Cu(B)$ that identifies $\Cu(A)$ with an ideal of $\Cu(B)$.
\end{lma}

%==========================================================================================
\begin{pgr}
\label{pgr:scaledProd}
Let $(S_j,\Sigma_j)_{j\in J}$ be a family of scaled \CuSgp{s}.
Let $(S,\Sigma)$ be their product in $\CatCuScaled$, which we also call their \emph{scaled product}. % of the family $(S_j,\Sigma_j)_{j\in J}$.

Let us describe $(S,\Sigma)$ more concretely.
The set-theoretic product $\prod_j\Sigma_j$ is a downward hereditary subset of $\CatPomProd_j S_j$.
Then
\[
\Sigma = \big\{ [(\vect{x}_t)_{t\leq 0}] \in \CatCuProd_{j\in J} S_j : \vect{x}_t \in \prod_{j \in J}\Sigma_j \text{ for every } t<0 \big\},
\]
and $S$ is the ideal of $\CatCuProd_{j} S_j$ generated by $\Sigma$.
Given $[(\vect{x}_t)_{t\leq 0}] \in \CatCuProd_{j} S_j$, we have $[(\vect{x}_t)_{t\leq 0}]\in S$ if and only if for every $t<0$ there exist $\sigma^{(1)},\ldots,\sigma^{(N)}\in\prod_j \Sigma_j$ such that $\vect{x}_t\llpw\sigma^{(1)}+\ldots+\sigma^{(N)}$.
\end{pgr}

%==========================================================================================
\begin{thm}
\label{thm:CuProd}
The scaled Cuntz semigroup functor preserves products.

More concretely, let $(A_j)_{j\in J}$ be a family of \ca{s}, and let $(S,\Sigma)$ be the scaled product of $(\CatCuScaled(A_j))_{j}$ with $\Sigma\subseteq S\subseteq \prod_j\Cu(A_j)$ as in \autoref{pgr:scaledProd}.
Then, the map $\Phi\colon \Cu(\prod_{j}A_j)\to \CatCuProd_{j} \Cu(A_j)$ from \autoref{pgr:CuProductSetup} is an order-embedding with image $S$ that moreover identifies the scale of $\Cu(\prod_{j\in J}A_j)$ with $\Sigma$:
\[
\CatCuScaled(\prod_{j\in J} A_J) = (\Cu(\prod_{j\in J}A_j), \Sigma_{\prod_j A_j}) \cong \CatCuProd_{j\in J} (\Cu(A_j),\Sigma_{A_j}) = (S,\Sigma).
\]
\end{thm}
\begin{proof}
Let $p\in\KK$ be a rank-one projection.
For each $j$, let $\iota_j\colon A_j\to A_j\otimes\KK$ be given by $\iota_j(a):=a\otimes p$.
This induces a natural map $\iota\colon \prod_j A_j \to \prod_j (A_j\otimes\KK)$, which identifies $\prod_j A_j$ with a hereditary sub-\ca{} of $\prod_j (A_j\otimes\KK)$.
By \autoref{prp:HerSubalg}, $\Cu(\iota)$ is an order-embedding whose image is an ideal.

Let $\Phi'$ be the natural map from \autoref{pgr:CuProductSetup} for the product $\prod_j(A_j\otimes\KK)$.
For each $j$, we have a \CuMor{} $\Cu(\iota_j)\colon\Cu(A_j)\to\Cu(A_j\otimes\KK)$, which induces a \CuMor{} $\alpha\colon\CatCuProd_j\Cu(A_j)\to\CatCuProd_j\Cu(A_j\otimes\KK)$.
Then
\[
\alpha\circ\Phi = \Phi'\circ\Cu(\iota),
\]
which means that the following diagram commutes:
\[
\xymatrix@R-10pt{
\Cu(\prod_{j\in J}A_j) \ar[d]^{\Cu(\iota)} \ar[r]^{\Phi}
& \CatCuProd_{j\in J} \Cu(A_j) \ar[d]^{\alpha}_{\cong} \\
\Cu(\prod_{j\in J}A_j\otimes\KK) \ar[r]_{\Phi'}^{\cong}
& \CatCuProd_{j\in J} \Cu(A_j\otimes\KK){~.}
}
\]
By \autoref{prp:CuProdStable}, $\Phi'$ is an order-isomorphism.
Further, each $\Cu(\iota_j)$ is an isomorphism, and consequently so is $\alpha$.
It follows that $\Phi$ is an order-embedding whose image is an ideal.
It remains to show that $\Phi$ maps $\Sigma_{\prod_j A_j}$ onto $\Sigma$.

First, let $x\in\Sigma_{\prod_j A_j}$.
By definition, for every $x'\ll x$, there exists a positive $\underline{a}\in \prod_j A_j$ with $x'\leq[\underline{a}]$.
Since $\Phi$ preserves suprema of increasing sequences, and since $\Sigma$ is downward hereditary, it is enough to verify that $\Phi([\underline{a}])\in\Sigma$ for every positive $\underline{a}\in\prod_j A_j$.
But this follows directly from the description of $\Phi$ in \autoref{pgr:CuProductSetup}.

Conversely, using that $\Phi$ is an order-embedding, and that $\Sigma_{\prod_j A_j}$ is closed under suprema of increasing sequences that moreover are preserved by $\Phi$, it is enough to verify that the image of $\Phi(\Sigma_{\prod_j A_j})$ is order-dense in $\Sigma$.
Thus, let $y',y\in\Sigma$ satisfy $y'\ll y$.
Choose an $\llpw$-increasing path $(\vect{y}_t)_{t\in(-\infty,0]}$ such that $y=[(\vect{y}_t)_{t\leq 0}]$, and choose $y_{t,j}\in\Cu(A_j)$ such that $\vect{y}_t=(y_{t,j})_j$.
Then choose $\varepsilon>0$ such that $y'\leq[(\vect{y}_{t-3\varepsilon})_{t\leq 0}]$.
By definition of $\Sigma$, we have $\vect{y}_t\in\CatPomProd_j\Sigma_{A_j}$ for each $t<0$.
In particular, $\vect{y}_{-\varepsilon}\in\CatPomProd_j\Sigma_{A_j}$.
For each $j\in J$, we have $y_{-2\varepsilon,j} \ll y_{-\varepsilon,j} \in \Sigma_{A_j}$.
By definition of $\Sigma_{A_j}$, we obtain $a_j\in (A_j)_+$ such that $y_{-2\varepsilon,j}\leq[a_j]$.

Applying \autoref{prp:PhiDense}, we obtain a positive, contractive $b_j\in A_j$ such that
\[
x_{-3\varepsilon,j} \ll [(b_j-\frac{1}{2})_+],\andSep
[b_j]=[a_j].
\]
Then $\underline{b}:=(b_j)_j$ is a positive element in $\prod_j A_j$.
Hence, the Cuntz class of $\underline{b}$ belongs to $\Sigma_{\prod_j A_j}$.
As in the proof of \autoref{prp:CuProdStable}, we obtain $y'\leq\Phi([\underline{b}])\leq y$.
\end{proof}

%==========================================================================================
In the case of finitely many \ca{s}, the map $\Phi$ from \autoref{thm:CuProd} is easily seen to be surjective, whence we obtain the following (well-known) result.
We note that the product of finitely many \CuSgp{s} is simply their set-theoretic product, equipped with componentwise order and addition;
see \autoref{prp:FinProdInCu}.

%==========================================================================================
\begin{cor}
\label{prp:CuProdFinite}
Let $(A_j)_{j\in J}$ be a \emph{finite} family of \ca{s}.
Then the map from \autoref{pgr:CuProductSetup} induces a natural isomorphism $\Cu(\prod_{j\in J}A_j)\cong\CatCuProd_{j\in J}\Cu(A_j)$.
\end{cor}

%==========================================================================================
\begin{exa}
\label{exa:CuProd}
Set $A_j:=\CC$ for each $j\in\NN$.
Set $A:=\prod_j A_j \cong \ell^\infty(\NN)$.
We have $\Cu(A_j)\cong\NNbar$ for each $j$, and the product $\CatCuProd_j\NNbar$ is defined as the equivalence classes of $\llpw$-increasing paths $(-\infty,0]\to \CatPomProd_j\NNbar$.
It follows that $\CatCuProd_j\NNbar$ may be identified with equivalence classes of componentwise increasing paths $(-\infty,0)\to \CatPomProd_j\NN$.
In particular, compact elements in $\CatCuProd_j\NNbar$ naturally corresponds to functions $\NN\to\NN$. 
(In \autoref{cor:cpcteltsinprods} we will generalize this to give a description of compact elements in arbitrary products.)

Let us see that the natural order-embedding $\Phi\colon \Cu(A) \to \CatCuProd_{j}\NNbar$ is not surjective. Recall that an element $x$ in a $\CatCu$-semigroup $S$ is said to be compact if $x\ll x$.
Note that $\Phi$ maps the Cuntz class of the unit of $A$ to the compact element in $\CatCuProd_{j}\NNbar$ that correspond to the function $f\colon \NN\to\NN$ with $f(j)=1$ for all $j$. 
Consider the compact element $x$ in $\CatCuProd_{j}\NNbar$ that corresponds to the function $g\colon\NN\to\NN$ with $g(j)=j$ for all $j$.
Since $g\nleq n f$ for every $n\in\NN$, it follows that $x\nleq \infty \Phi([1_A])$.
In particular, $\Phi([1_A])$ is not full, and $\Cu(\prod_j A_j)$ is isomorphic to a proper ideal of $\prod_j\Cu(A_j)$.
\end{exa}

%===============================================================
It is natural to ask to what extent the Cuntz semigroup functor preserves arbitrary limits.
We are thankful to Eusebio Gardella for pointing out this example.

%==========================================================================================
\begin{exa}
\label{exa:suzuki}
Let $\Gamma$ be a countable, discrete, exact group satisfying the approximation property.
(For instance, any countable, discrete, amenable group.)
In \cite[Theorem A]{Suz17GpAlgDecrIntersection}, Suzuki constructs a decreasing sequence of \ca{s} $(A_n)_{n\in\NN}$, each of which is isomorphic to $\mathcal{O}_2$, whose intersection is isomorphic to the reduced group \ca{} $C_r^*(\Gamma)$.

It follows that $C_r^*(\Gamma)\cong\varprojlim A_n$.
For each $n\in\NN$, we have $\Cu(A_n)\cong\{0,\infty\}$.
Since the maps in the inverse system given by the algebras $A_n$ are defined by the inclusions $A_n\subseteq A_{n-1}$, at the level of the Cuntz semigroup they induce the identity maps.
It follows that $\CatCuinvLim\limits\Cu(A_n)\cong\{0,\infty\}$.
However, $\Cu(C_r^*(\Gamma))\not\cong\{0,\infty\}$ since, for example, $C_r^*(\Gamma)$ has a normalized trace.
\end{exa}

%==========================================================================================
\begin{rmk}
One can consider a product over an infinite index set $J$ as the inverse limit of the subproducts of finite subsets of $J$.
By \autoref{prp:CuProdStable}, the Cuntz semigroup functor preserves products of stable \ca{s}, and therefore it \emph{does} preserve this particular inverse limit.
One difference between \autoref{exa:suzuki} and a product of \ca{s} is that, in the latter, all connecting maps are surjective.
Therefore, the following question is pertinent:
\end{rmk}

%==========================================================================================
\begin{qst}
\label{qst:invLimSurj}
Does the Cuntz semigroup functor preserve inverse limits with surjective connecting morphisms?
That is, do we have
\[
\Cu(\varprojlim A_n)\cong \CatCuinvLim\Cu(A_n)
\]
for every (sequential) inverse system $(A_n)_{n\in\NN}$ of (stable) \ca{s} where each map $A_{n+1}\to A_n$ is surjective?
\end{qst}

%==========================================================================================
%==========================================================================================
\section{Cuntz semigroups of direct sums and sequence algebras}
\label{sec:dirsum}

%==========================================================================================
In this section we show that the Cuntz semigroup functor preserves direct sums.
It follows that the scaled Cuntz semigroup functor preserves sequence algebras;
see \autoref{prp:CuSequenceAlg}.

%==========================================================================================
In $\CatCu$, product and coproduct of finite families are naturally isomorphic;
see Propositions~\ref{prp:FinProdInCu} and~\ref{prp:FinCoprodInCu}.
It follows that for any family $(S_j)_{j\in J}$ of \CuSgp{s}, the direct sum $\bigoplus_{j\in J}S_j$ is naturally isomorphic to the coproduct $\coprod_{j\in J} S_j$;
see \autoref{pgr:WfinCoproduct}.
The analogous statement holds in $\CatCuScaled$.
In the category of \ca{s}, it is not true that finite products and coproducts coincide (see \autoref{ex:coproductca}).
Nevertheless, there is a zero object and the categorical construction of the direct sum yields the usual definition of direct sums for \ca{s}.

%==========================================================================================
\begin{thm}
\label{thm:coproduct}
Let $(A_j)_{j\in J}$ be a family of \ca{s}.
Then there is are natural isomorphisms
\[
\Cu\big( \bigoplus_{j\in J}A_j \big) \cong \CatCuCoprod_{j\in J}\Cu(A_j), \andSep
\CatCuScaled\big( \bigoplus_{j\in J}A_j \big) \cong \CatCuCoprod_{j\in J}\CatCuScaled(A_j).
\]
\end{thm}
\begin{proof}
We use that direct sums are defined as inductive limits of finite products.
Using that the Cuntz semigroup functor preserves inductive limits (\cite[Corollary~3.2.9, p.29]{AntPerThi18TensorProdCu}) and finite products (\autoref{prp:CuProdFinite}), we obtain the first isomorphism.
The second isomorphism follows analogously, using \autoref{prp:CuScaledIndLim}.
\end{proof}

%==========================================================================================
\begin{exa}
\label{ex:coproductca}
The Cuntz semigroup functor does not preserve coproducts.
Given two \ca{s} $A$ and $B$, their coproduct in the category of \ca{s} is the (full) free product $A\ast B$.
Consider for example $A=B=\CC$.
It is known that
\[
\CC\ast\CC \cong \big\{ f\in C([0,1],M_2(\CC)) :
f(0) \text{ diagonal}, 
f(1)= \begin{psmallmatrix}
x & 0 \\ 0 & 0 \\
\end{psmallmatrix},
\text{ some } x\in\CC \big\},
\]
see \cite[Example~IV.1.4.2, p.330]{Bla06OpAlgs}.
We have $\Cu(\CC)\cong\NNbar$, and the coproduct of $\NNbar$ and $\NNbar$ in the category $\CatCu$ is simply $\NNbar\times\NNbar$ with coordinatewise order and addition.
On the other hand, it is easy to see that $\Cu(\CC\ast\CC) \ncong \NNbar\times\NNbar$.
(Using \cite[Corollary~3.5]{AntPerSan11PullbacksCu}, one can in fact explicitly compute $\Cu(\CC\ast\CC)$.)
\end{exa}

%==========================================================================================
\begin{ntn}
\label{ntn:c0}
Let $(S_j)_{j\in J}$ be a family of \CuSgp{s}.
Given $\vect{x}=(x_j)_{j\in J}$ in $\CatPomProd_{j\in J}S_j$, we define its \emph{support} as $\supp(\vect{x}) := \{ j\in J : {x}_j\neq 0 \}$.
Set 
\[
\cc_0\big( (S_j)_{j} \big) :=\big\{ [(\vect{x}_t)_{t\leq 0}] \in \CatCuProd_jS_j :  \supp(\vect{x}_t) \text{ is finite, for every } t<0 \big\}.
\]
Given a \CuSgp{} $S$, we write $\cc_0(S)$ for $\cc_0\big( (S)_{n\in\NN} \big)$.
\end{ntn}

%==========================================================================================
\begin{prp}
\label{lma:coproductinproduct}
Let $(S_j)_{j\in J}$ be a family of \CuSgp{s}.
Then $\cc_0\big( (S_j)_{j} \big)$ is an ideal in $\prod_{j}S_j$, and the canonical map $\CatCuCoprod_{j}S_j\to \prod_{j}S_j$ is an order-embedding whose image is $\cc_0\big( (S_j)_{j} \big)$.
\end{prp}
\begin{proof}
For each $F\Subset J$, we let $\varphi_{J,F}\colon\prod_{j\in F}S_j\to \prod_{j\in J} S_j$ be the natural morphism.
Since these maps are compatible with the inductive system defining $\CatCuCoprod_{j}S_j$, they induce a natural morphism $\varphi\colon \CatCuCoprod_{j}S_j = \varinjlim_{F\Subset J}\prod_{j\in F} S_j\to \prod_{j\in J}S_j$. 
Using that each $\varphi_{J,F}$ is an order-embedding, it follows that $\varphi$ is an order-embedding as well.
To simplify, we write $\cc_0$ for $\cc_0\big( (S_j)_{j} \big)$.
It is easy to verify that $\cc_0$ is an ideal in $\CatCuProd_jS_j$.
Since the image of each $\varphi_{J,F}$ is contained in $\cc_0$, so is the image of $\varphi$.

To show that the image of $\varphi$ is $\cc_0$, let $x\in\cc_0$.
Choose $\vect{x}_t=(x_{t,j})_j\in\CatPomProd_j S_j$ such that $x=[(\vect{x}_t)_{t\leq 0}]$.
For $n\geq 1$, set $F_n:=\supp(\vect{x}_{-1/n})$.
Since $x\in\cc_0$, we obtain that $F_n$ is finite.
For $t\leq 0$ set
\[
\vect{x}^{(n)}_t := ( x_{t-\tfrac{1}{n},j} )_{j\in F_n} \in \CatPomProd_{j\in F_n} S_j.
\]
Then $(\vect{x}^{(n)}_t)_{t\leq 0}$ is a path, and we let $x^{(n)}$ denote its class in $\CatCuProd_{j\in F_n} S_j$.
By construction, we have $\varphi_{J,F}(x^{(n)}) = [(\vect{x}_{t-1/n})_{t\leq 0}]$.
The images of $x^{(n)}$ in $\CatCuCoprod_{j}S_j$ form an increasing sequence, and we let $y$ denote their supremum.
Then
\[
\varphi(y) = \sup_n \varphi_{J,F}(x^{(n)}) 
= \sup_n [(\vect{x}_{t-1/n})_{t\leq 0}]
= [(\vect{x}_t)_{t\leq 0}]
= x. \qedhere
\]
\end{proof}

%==========================================================================================
The following result is straightforward to prove.
We omit the details.

%==========================================================================================
\begin{prp}
\label{prp:coproductScaled}
Let $(S_j,\Sigma_j)_{j\in J}$ be a family of scaled \CuSgp{s}, and let $(S,\Sigma)$ be their scaled product with $\Sigma\subseteq S\subseteq \prod_j S_j$ as in \autoref{pgr:scaledProd}. %, and let $\cc_0\subseteq\prod_j S_j$ be as in \autoref{lma:coproductinproduct}.
Then $\cc_0\big( (S_j)_{j} \big)\subseteq S$, and we have a canonical isomorphism:
\[
\bigoplus_{j\in J} (S_j,\Sigma_j) 
\cong \left( \cc_0\big( (S_j)_{j} \big), \cc_0\big( (S_j)_{j} \big)\cap \Sigma \right).
\]
\end{prp}

%==========================================================================================
Recall that the \emph{sequence algebra} of a \ca{} $A$ is defined as 
\[
A_\infty := \ell^\infty (A)/c_0(A)
\cong \frac{\prod_{n\in \NN} A}{\bigoplus_{n\in \NN} A}.
\]

By combining Theorems~\ref{thm:coproduct} and~\ref{thm:CuProd}, \autoref{lma:coproductinproduct}, \cite[Theorem~5]{CiuRobSan10CuIdealsQuot} we can now describe the Cuntz semigroup of the sequence algebra:

%==========================================================================================
\begin{cor}
\label{prp:CuSequenceAlg}
Let $A$ be a \ca.
Let $(S,\Sigma)$ be the scaled product of the family $(\Cu(A))_{n\in\NN}$, with $\Sigma\subseteq S\subseteq \prod_n \Cu(A)$ as in \autoref{pgr:scaledProd}.
Then 
\[
\Cu_{sc}(A_\infty)
\cong S / \cc_0(\Cu(A)).
\]
\end{cor}

%==========================================================================================
%==========================================================================================
\section{Cuntz semigroups of ultraproducts}
\label{sec:ultraprod}

In this section, we show that the scaled Cuntz semigroup functor preserves ultraproducts;
see \autoref{thm:CuUltraprod}.
In particular, the Cuntz semigroup functor preserves ultraproducts of \emph{stable} \ca{s};
see \autoref{prp:CuUltraprodStable}.
We take the opportunity to offer a more categorical view on ultraproducts, which suits both \ca{s} and $\CatCu$-semigroups.

%==========================================================================================
%==========================================================================================
\subsection{Categorical ultraproducts}

%==========================================================================================
We first recall the basic theory of categorical ultraproducts.

%==========================================================================================
\begin{pgr}
\label{pgr:ultraprCat}
Let $\CatC$ be a category that has (small) products and inductive limits, let $J$ be a set, let $\filter$ be an ultrafilter on $J$, and let $(X_j)_{j\in J}$ be a family of objects in $\CatC$.
Let $G\subseteq F\subseteq J$.
By the universal property of the product we obtain a morphism
\[
\varphi_{G,F}\colon\prod_{j\in F}X_j \to \prod_{j\in G}X_j,
\]
such that $\pi_{j,F}=\pi_{j,G}\circ\varphi_{G,F}$ for each $j\in G$.

We order the elements of $\filter$ by reversed inclusion.
Then $\filter$ is upward directed, and we obtain an inductive system indexed over $\filter$, with objects $\prod_{j\in F}X_j$ for $F\in\filter$, and with morphism $\varphi_{G,F}$ for $F,G\in\filter$ with $F\supseteq G$.
The inductive limit of this system is called the (categorical) \emph{ultraproduct} of $(X_j)_{j\in J}$ along $\filter$:
\[
\prod_{\filter} X_j := \varinjlim_{F\in\filter} \prod_{j\in F}X_j.
\]
We let $\pi_\filter\colon \prod_{j\in J}X_j \to \prod_\filter X_j$ denote the natural morphism to the inductive limit.

The ultraproduct of a constant family of objects is called \emph{ultrapower}.
Given an object $X$, we denote its ultrapower by $X_\filter$, that is, $X_\filter := \prod_\filter X$.

The ultraproduct construction is functorial in the following sense:
Given morphisms $\alpha_j\colon X_j\to Y_j$, we obtain a natural morphism $\alpha_\filter\colon\prod_\filter X_j \to \prod_\filter Y_j$ as follows:
For every $F\in\filter$, the universal property of products implies that the morphisms $(\alpha_j)_{j\in F}$ induce a natural morphism $\alpha_F\colon\prod_{j\in F}X_j \to \prod_{j\in F}Y_j$.
These morphisms then induce a morphism between the ultraproducts, defined as respective inductive limits.
Similarly, a morphism $X\to Y$ induces a morphism $X_\filter\to Y_\filter$.
\end{pgr}

%==========================================================================================
\begin{pgr}
\label{pgr:FunctorUltraproduct}
Let $\CatC$ and $\CatD$ be categories that have products and inductive limits (and hence categorical ultraproducts), and let $F\colon\CatC\to\CatD$ be a functor that preserves inductive limits.
Let $\filter$ be an ultrafilter on a set $J$, and let $(X_j)_{j\in J}$ be a family of objects in~$\CatC$.
Then there is a natural morphism
\[
\Phi_\filter\colon F(\prod_\filter X_j) \to \prod_\filter F(X_j).
\]
If the functor $F$ also preserves products, then $\Phi_\filter$ is a natural isomorphism.
\end{pgr}

%==========================================================================================
\begin{pgr}
\label{pgr:CuUltraproductSetup}
Let $\filter$ be an ultrafilter on a set $J$, and let $(A_j)_{j\in J}$ be a family of \ca{s}.
The category of \ca{s}, and the categories $\CatCu$ and $\CatCuScaled$ of (scaled) \CuSgp{s} are complete and cocomplete, and thus each of the categories admit categorical ultraproducts in the sense of \autoref{pgr:ultraprCat}.
By \cite[Corollary~3.2.9, p.29]{AntPerThi18TensorProdCu} and \autoref{prp:CuScaledIndLim}, the (scaled) Cuntz semigroup functor preserves arbitrary inductive limits.
(It was shown in \cite{CowEllIva08CuInv} that the Cuntz semigroup functor preserves \emph{sequential} inductive limits;
for the application to ultraproducts it is however crucial that the functor also preserves inductive limits over non-sequential directed sets.)
As explained in \autoref{pgr:FunctorUltraproduct}, we obtain natural (scaled) \CuMor{s}
\[
\Phi_\filter\colon \Cu(\prod_\filter A_j) \to \prod_\filter \Cu(A_j), \andSep
\Phi_{\filter,\mathrm{sc}}\colon \CatCuScaled(\prod_\filter A_j) \to \prod_\filter \CatCuScaled(A_j).
\]
Applying \autoref{prp:CuProdStable} and \autoref{thm:CuProd}, we obtain the following results:
%deduce that the (scaled) Cuntz semigroup functor preserves ultraproducts of (stable) \ca{s}:
\end{pgr}

%==========================================================================================
\begin{prp}
\label{prp:CuUltraprodStable}
%The Cuntz semigroup functor preserves ultraproducts of \emph{stable} \ca{s}:
Let $\filter$ be an ultrafilter on a set $J$, and let $(A_j)_{j\in J}$ be a family of stable \ca{s}.
Then, the map $\Phi_\filter\colon \Cu(\prod_\filter A_j) \to \prod_\filter \Cu(A_j)$ from \autoref{pgr:CuUltraproductSetup} is an isomorphism.
\end{prp}

%==========================================================================================
\begin{thm}
\label{thm:CuUltraprod}
The scaled Cuntz semigroup functor preserves ultraproducts:
Given an ultrafilter $\filter$ on a set $J$, and a family $(A_j)_{j\in J}$ of \ca{s}, the map $\Phi_{\filter,\mathrm{sc}}$ from \autoref{pgr:CuUltraproductSetup} is an isomorphism:
\[
\CatCuScaled(\prod_\filter A_j) \cong \prod_\filter \CatCuScaled(A_j).
\]
\end{thm}

%==========================================================================================
%==========================================================================================
\subsection{A different picture of ultraproducts}

%==========================================================================================
In many `everyday' categories, the (categorical) ultraproduct admits a more concrete description, which is the one used in practice.
We will first recall this for the category of \ca{s}, and then obtain a similar result for the category $\CatCu$.

%==========================================================================================
\begin{pgr}
\label{pgr:ultraprCAlg}
Let $\filter$ be an ultrafilter on a set $J$, and let $(A_j)_{j\in J}$ be a family of \ca{s}. 
Set $A:=\prod_{j\in J}A_j$.
We define the support of an element $\underline{a}=(a_j)_j\in A$ as $\supp(\underline{a}) := \{ j\in J : a_j \neq 0 \}$.
Given $F\in\filter$, we set
\[
I_F := \big\{ a\in A : \supp(a)\cap F =\emptyset \big\}.
\]
Then $I_F$ is a (closed, two-sided) ideal in $A$ that is naturally isomorphic to $\prod_{j\in F^c}A_j$.
(We use the convention that the product over the empty set is the zero \ca{}.)
Moreover, the canonical map $\varphi_{F,J}\colon A \to \prod_{j\in F}A_j$ is a surjective \stHom{} with kernel $I_F$.
Thus, we obtain a natural isomorphism $\prod_{j\in F}A_j \cong A/I_F$.

Given $F\supseteq G$, we have $I_F\subseteq I_G$.
Under the isomorphisms $\prod_{j\in F}A_j\cong A/I_F$ and $\prod_{j\in G}A_j\cong A/I_G$, the connecting morphism $\varphi_{G,F}$ (defined as in \autoref{pgr:ultraprCat}) corresponds to the natural quotient map $A/I_F \to A/I_G$.
Let $I$ be the ideal of $A$ generated by the $I_F$, that is,
\[
I := \overline{\bigcup_{F\in\filter} I_F}^{\|\cdot\|} \subseteq A.
\] 

The inductive limit of the system $A/I_F$ is isomorphic to $A/I$.
It follows that the ultraproduct $\prod_{\filter} A_j$ is naturally isomorphic to $A/I$, as shown in the following commutative diagram:
\[
\xymatrix@R-10pt{
A/I_F \ar[d]^{\cong} \ar[r]
& A/I_G \ar[d]^{\cong} \ar[r]
& \ldots \ar[r]
& \varinjlim_{F\in\filter} A/I_F \ar[d]^{\cong} \ar@{}[r]|-{=}
& A/I \ar[d]^{\cong} \\
\prod_{j\in F}A_j \ar[r]^{\varphi_{G,F}}
& \prod_{j\in G}A_j \ar[r]
& \ldots \ar[r]
& \varinjlim_{F\in\filter} \prod_{j\in F}X_j \ar@{}[r]|-{=}
& \prod_{\filter} A_j.
}
\]

It remains to describe the ideal $I$.
%Let $\beta J$ denote the Stone-\Cech{} compactification of $J$. 
%The ultrafilter $\filter$ corresponds to a point in $\beta J$. 
%Given a compact, Hausdorff space $X$, each map $f\colon J\to X$ corresponds uniquely to a continuous map $\tilde{f}\colon \beta J \to X$.
%The value $\tilde{f}(\filter)$ is called the limit of $f$ along $\filter$ and is denoted by $\lim_{j\to\filter}f(j)$.
Given $(a_j)_j\in A=\prod_j A_j$, we have a bounded function $J\to\RR$, $j\mapsto\|a_j\|$, which allows us to consider $\lim_{j\to\filter}\|a_j\|$.
For $a=(a_j)_j\in A$ we have $\lim_{j\to\filter}\|a_j\|=0$ if and only if $\{j\in J : \|a_j\|<\varepsilon\}\in\filter$ for every $\varepsilon>0$.
It follows that
\[
I = \big\{ (a_j)_j \in A : \lim_{j\to\filter} \|a_j\| =0 \big\}.
\]
This shows that the `category theoretic' ultraproduct of \ca{s} agrees with the notion of ultraproduct of \ca{s} considered in (continuous) model theory;
see for example \cite{GeHad01Ultraproducts}.
The ideal $I$ is also denoted by $c_\filter((A_j)_j)$.

For every positive element $a_j\in A_j$ and $\varepsilon>0$ we have $\| a_j \| \leq \varepsilon$ if and only if $(a_j-\varepsilon)_+=0$.
Thus,
\[
I_+ = \big\{ (a_j)_j\in A_+ : \supp ((a_j-\varepsilon)_+)_{j\in J})\notin\filter \text{ for every } \varepsilon>0 \big\}.
\]
\end{pgr}

%==========================================================================================
\begin{pgr}
\label{pgr:ultraprCu}
Let $\filter$ be an ultrafilter on a set $J$, and let $(S_j)_{j\in J}$ be a family of \CuSgp{s}.
Set $P:=\CatPomProd_j S_j$, the set-theoretic product of the $S_j$.
The product $S:=\prod_{j\in J}S_j$ in $\CatCu$ is given by equivalence classes of $\llpw$-increasing paths $(-\infty,0]\to P$.
Given $F\in\filter$, we set
\[
I_F := \big\{ [(\vect{x}_t)_{t\leq0}] \in S \text{ such that }  \supp(\vect{x}_0)\cap F =\emptyset \big\}.
\]
Then $I_F$ is an ideal of $S$ that is canonically isomorphic to $\prod_{j\in F^c}S_j$.
As for \ca{s}, the map $\varphi_{F,J}\colon S \to \prod_{j\in F}S_j$ induces an isomorphism $\prod_{j\in F}S_j\cong S/I_F$.
Let $I$ be the ideal of $S$ generated by the $I_F$, that is,
\[
I :=  \overline{\bigcup_{F\in\filter} I_F}^{\sup} 
= \big\{ \sup_n x_n : (x_n)_n \text{ increasing sequence in } \bigcup_{F\in\filter}I_F \big\}
\subseteq S.
\]

We obtain natural isomorphisms
\[
\prod_{\filter} S_j := \varinjlim_{F\in\filter}\prod_{j\in F}S_j 
\cong \varinjlim_{F\in\filter} S/I_F
\cong S/I.
\]
Let $\pi_\filter\colon\prod_{j\in J}S_j\to\prod_\filter S_j$ denote the quotient map with kernel $I$.
The next result describes the ideal $I$.
\end{pgr}

%==========================================================================================
\begin{lma}
\label{prp:ultraprodIdeal}
We retain the notation from \autoref{pgr:ultraprCu}.
Let $(\vect{x}_t)_{t\leq 0}$ be a path in $(\CatPomProd_j S_j,\llpw	)$.
Then the following are equivalent:
\begin{enumerate}
\item 
We have $[(\vect{x}_t)_{t\leq 0}] \in I$.
\item 
For every $\varepsilon>0$, there exists $F\in\filter$ such that the $\varepsilon$-cut-down
$[(\vect{x}_{t-\varepsilon})_{t\leq 0}]$ belongs to $I_F$.
\item 
For every $t<0$, we have $\supp( \vect{x}_t ) \notin \filter$.
\end{enumerate}
\end{lma}
\begin{proof}
(1)$\Rightarrow$(2):
Suppose $[(\vect{x}_t)_{t\leq 0}] \in I$.
Then there is an increasing sequence $(x_n)_n$ in $\bigcup_{F\in\filter} I_F$ such that $[(\vect{x}_t)_{t\leq 0}] =\sup x_n$.
Let $\varepsilon>0$.
Using that $[(\vect{x}_{t-\varepsilon})_{t\leq 0}]\ll [(\vect{x}_t)_{t\leq 0}]$, we obtain $n$ such that $[(\vect{x}_{t-\varepsilon})_{t\leq 0}]\leq x_n$.
Since $x_n\in\bigcup_{F\in\filter} I_F$ and since each $I_F$ is an ideal of $S$, there exists $F\in\filter$ such that $[(\vect{x}_{t-\varepsilon})_{t\leq 0}]$ belongs to $I_F$.

(2)$\Rightarrow$(3):
Let $t<0$.
Set $\varepsilon:=-t$.
By assumption, there exists $F\in\filter$ such that $[(\vect{x}_{s-\varepsilon})_{s\leq 0}]$ belongs to $I_F$.
By definition, this means $\supp(\vect{x}_{0-\varepsilon})\cap F=\emptyset$, and thus $\supp(\vect{x}_{t})\notin\filter$.

(3)$\Rightarrow$(1): Suppose that for every $t<0$, we have $\supp(\vect{x}_t)\notin\filter$. Now, write 
\[
[(\vect{x}_t)_{t\leq 0}]=\sup\limits_{n\geq 1} [(\vect{x}_{t-\frac{1}{n}})_{t\leq 0}],
\]
and put $F_n=\supp (\vect{x}_{-\frac{1}{n}})^c$. It is clear now that $F_n\in\filter$ and that $[(\vect{x}_{t-\frac{1}{n}})_{t\leq 0}]\in I_{F_n}$. Therefore $[(\vect{x}_t)_{t\leq 0}]\in I$, as desired.
\end{proof}

%==========================================================================================
\begin{ntn}
\label{ntn:cfilter}
Let $\filter$ be an ultrafilter on a set $J$, and let $(S_j)_{j\in J}$ be a family of \CuSgp{s}.
Set
\[
\cc_\filter\big( (S_j)_j \big) := \big\{ [(\vect{x}_t)_{t\leq 0}] \in \CatCuProd_j S_j : \supp( \vect{x}_t ) \notin\filter \text{ for each } t<0  \big\}.
\]
\end{ntn}

%==========================================================================================
\begin{prp}
\label{cor:cu}
We retain the notation from \autoref{pgr:ultraprCu}.
%Let $\filter$ be an ultrafilter on a set $J$, and let $(S_j)_{j\in J}$ be a family of \CuSgp{s}.
Then $\cc_\filter\big( (S_j)_j \big)$ is an ideal in $\CatCuProd_j S_j$, and we have a natural isomorphism 
\[
\prod_{\filter} S_j \cong \big(\prod_{j\in J}S_j\big) / c_\filter((S_j)_j).
\]
\end{prp}

%==========================================================================================
\begin{thm}
\label{prp:CuUltraprodConcrete}
Let $\filter$ be an ultrafilter on a set $J$, and let $(A_j)_{j\in J}$ be a family of \ca{s}.
Let $(S,\Sigma)$ be the scaled product of $(\CatCuScaled(A_j))_{j}$ with $\Sigma\subseteq S\subseteq \prod_j\Cu(A_j)$ as in \autoref{pgr:scaledProd}.
Then we have a commutative diagram, where each row is a short exact sequence, where the upper vertical maps are isomorphisms, and where the lower vertical maps are inclusions as ideals:
\[
\xymatrix@R-10pt{
\Cu(\cc_\filter((A_j)_j)) \ar[d]_{\cong} \ar@{^{(}->}[r]
& \Cu( \prod_j A_j ) \ar[d]^{\Phi}_{\cong} \ar@{->>}[r]
& \Cu( \prod_\filter A_j ) \ar[d]^{\Phi_\filter}_{\cong} \\
S\cap \cc_\filter\big((\Cu(A_j))_{j\in J}\big) \ar[d] \ar@{^{(}->}[r]
& S \ar[d] \ar@{->>}[r]
& S/\left( S\cap \cc_\filter\big((\Cu(A_j))_{j\in J}\big) \right) \ar[d] \\
\cc_\filter\big((\Cu(A_j))_{j\in J}\big) \ar@{^{(}->}[r]
& \prod_j\Cu(A_j) \ar@{->>}[r]
& \prod_\filter \Cu(A_j).
}
\]
In particular, the map $\Phi_\filter$ from \autoref{pgr:CuUltraproductSetup} is an order-embedding that identifies $\Cu(\prod_\filter A_j)$ with the image of $S$ under the map $\pi_\filter$ from \autoref{pgr:ultraprCu}.
Further, $\Phi_\filter$ identifies the scale of $\Cu(\prod_\filter A_j)$ with $\pi_\filter(\Sigma)$.
\end{thm}
\begin{proof}
In general, if $J$ is an ideal in a \ca{} $A$, then the inclusion $J\to A$ induces a map that identifies $\Cu(J)$ with an ideal of $\Cu(A)$, and such that $\Cu(A)/\Cu(J)$ is naturally isomorphism to $\Cu(A/J)$;
see \cite[Theorem~4.1]{CiuRobSan10CuIdealsQuot}, and note that the assumption that the ideal be $\sigma$-unital can be removed;
see also \cite[Section~5.1, p.37ff]{AntPerThi18TensorProdCu}.
This shows that the top row in the diagram is exact.

By \autoref{thm:CuProd}, $\Phi$ is an order-embedding with image $S$.
To simplify, we denote $\cc_\filter((A_j)_j)$ by $\cc_\filter^A$, and we denote $\cc_\filter\big((\Cu(A_j))_{j\in J}\big)$ by $\cc_\filter^S$. 

Claim:
Let $\underline{a} \in (\prod_j A_j)_+$.
Then $\Phi([\underline{a}])\in\cc_\filter^S$ if and only if $\underline{a}\in\cc_\filter^A$.

To prove the claim, let $a_j$ such that $\underline{a}=(a_j)_j$, and set $\vect{x}_t := ([(a_j+t)_+])_j$ for $t\leq 0$.
Then $\Phi([\underline{a}]) = [(\vect{x}_t)_{t\leq 0}]$;
see \autoref{pgr:CuProductSetup}.
As noted at the end of \autoref{pgr:ultraprCAlg}, we have $\underline{a}\in c_\filter^A$ if and only if $\supp (((a_j-\varepsilon)_+)_{j\in J})\notin\filter$ for every $\varepsilon>0$, which in turn is equivalent to $\supp(\vect{x}_t)\notin\filter$ for every $t<0$.
Now, the claim follows from \autoref{prp:ultraprodIdeal}.

It follows that for $x\in \Sigma_{\prod_j A_j}$ we have $x\in\Cu(\cc_\filter((A_j)_j))$ if and only if $\Phi(x)\in\cc_\filter^S$.
Since $\Phi(\Sigma_{\prod_j A_j})=\Sigma$, it follows that $\Phi( \Cu(\cc_\filter^A) \cap \Sigma_{\prod_j A_j} )=\cc_\filter^S \cap \Sigma$.
Using \autoref{pgr:scalesIdealsQuotients} at the third step, we obtain 
\begin{align*}
\Phi( \Cu(\cc_\filter^A) )
= \Phi( \langle \Sigma_{\cc_\filter^A} \rangle )
= \langle \Phi( \Sigma_{\cc_\filter^A} ) \rangle
= \langle \Phi( \Cu(\cc_\filter^A) \cap \Sigma_{\prod_j A_j} ) \rangle
= \langle \cc_\filter^S \cap \Sigma \rangle
= \cc_\filter^S \cap S,
\end{align*}
where $\langle\freeVar\rangle$ denotes the generated ideal.
It follows that the upper vertical maps in the diagram are isomorphisms.

Lastly, to show the statement about the scales, let $\psi\colon \Cu(\prod_j A_j)\to \Cu(\prod_\filter A_j)$ be the \CuMor{} induced by the quotient \stHom.
Using \autoref{pgr:scalesIdealsQuotients} at the last step, we get
\[
\pi_\filter(\Sigma)
= \pi_\filter(\Phi(\Sigma_{\prod_j A_j}))
= \Phi_\filter( \psi( \Sigma_{\prod_j A_j} ))
= \Phi_\filter( \Sigma_{\prod_\filter A_j} ). \qedhere
\]
\end{proof}

%==========================================================================================
\begin{exa}
\label{prp:prodE0}
Let $J$ be an infinite set, and let $\filter$ be an ultrafilter on $J$.
We identify the Stone \Cech\ compactification $\beta(J)$ of $J$ with the set of ultrafilters on $J$, and we equip it with the \emph{Stone topology}, which has a basis of clopen sets given by 
\[
\mathcal{U}_A := \big\{ \mathcal{F}\in\beta(J) : A\in\mathcal{F} \big\},
\]
for each $A\subseteq J$.
Let $\mathrm{cpct}_\sigma(\beta(J))$ denote the family of $\sigma$-compact, open subsets of $\beta(J)$.
Note that $\mathrm{cpct}_\sigma(\beta(I))$ is a \pom{}, where addition of sets is given by their union, and order is defined by inclusion. 
Then we have isomorphisms of \CuSgp{s}:
\[
\CatCuProd_{j\in J} \{0,\infty\} 
\cong \mathrm{cpct}_\sigma(\beta(J)),\andSep
\CatCuProd_\filter\{0,\infty\} \cong \{0,\infty\}.
\]
\end{exa}
\begin{proof}
Given $A,B\subseteq J$, we have $\mathcal{U}_{A\cup B}=\mathcal{U}_A\cup\mathcal{U}_B$, and 
$\mathcal{U}_A\subseteq\mathcal{U}_B $ if and only if $A\subseteq B$.
We first compute $\CatCuProd_J \{0,\infty\}$.
Let $\mathcal{P}(J)$ denote the powerset of $J$ considered as a \pom{} with addition given by union and order given by inclusion.
It is clear that $\CatPomProd_j\{0,\infty\}\cong \mathcal{P}(J)$.
Under this identification, the auxiliary relation $\llpw$ becomes inclusion and it is easy to see that in this case the $\tau$-construction agrees with the (sequential) round ideal completion $\gamma$.
Hence,
\[
\CatCuProd\limits_{j\in J}\{0,\infty\}
= \tau\big( \CatPomProd_{j\in J}\{0,\infty\},\llpw \big)
\cong \tau\big( \mathcal{P}(J),\subseteq \big)
\cong \gamma\big( \mathcal{P}(J),\subseteq \big).
\]

The elements in $\gamma(\mathcal{P}(J),\subseteq)$ are equivalence classes of increasing sequences $\mathbf{A}=(A_n)_{n\in\NN}$ of subsets of $J$, under the antisymmetrization of the relation $(A_n)_n\precsim (B_n)_n$ if for every $n\in\NN$ there is $m\in\NN$ such that $A_n\subseteq B_m$. We use $[\mathbf{A}]$ to denote the class of $\mathbf{A}$, and recall that addition is given by $[(A_n)_n]+[(B_n)_n]=[(A_n\cup B_n)_n]$.
Define $\Phi\colon \gamma(\mathcal{P}(J),\subseteq)\to \mathrm{cpct}_\sigma(\beta(J))$ by
\[
\Phi([\mathbf{A}]) := \bigcup\limits_{n\in\NN}\mathcal{U}_{A_n},
\]
for every increasing sequence $\mathbf{A}=(A_n)_{n\in\NN}$. Using the properties of the clopen sets $\mathcal U_A$ mentioned above, it follows that $\Phi$ is well-defined, order-preserving and additive.

To see that $\Phi$ is surjective, let $U\subseteq \beta(J)$ be a $\sigma$-compact, open subset. Choose an increasing sequence $(K_n)_n$ of compact subsets such that $U=\bigcup_n K_n$. Since the $\mathcal U_A$ form a basis of the topology, and the $K_n$ are compact, we can find sets $D_n$ such that $K_n\subseteq \mathcal U_{D_n} \subseteq U$. Therefore, setting $A_n=\bigcup_{i=1}^n D_i$ to make them increasing, we have $U=\bigcup_n \mathcal U_{A_n}$ and  $U=\Phi([\vect A])$ where $\vect{A}=(A_n)_n$. 

Now suppose that $\Phi([\mathbf{A}])\leq\Phi([\mathbf{B}])$.
This implies that for each $n$,  $\mathcal{U}_{A_n}\subseteq\bigcup_{m\in\NN}\mathcal{U}_{B_m}$.
Since $\mathcal{U}_{A_n}$ is compact and the sequence $(B_m)_m$ is increasing, there is $m\in\NN$ such that $\mathcal{U}_{A_n}\subseteq\mathcal{U}_{B_m}$ which implies $A_n\subseteq B_m$.
Hence $\vect{A}\precsim \vect{B}$ proving that $\Phi$ is an order-embedding, and thus $\CatCuProd_{j} \{0,\infty\} 
\cong \mathrm{cpct}_\sigma(\beta(J))$.

We now compute the ultraproduct.
Let us write $c_\filter$ for $c_\filter((\{0,\infty\})_j)$.
It follows from \autoref{cor:cu} that $c_\filter$ corresponds to equivalence classes of $\mathbf{A}=(A_n)_n$ with all $A_n\not\in \filter$.
Let $[\mathbf{A}],[\mathbf{B}]\in\gamma(\mathcal{P}(J),\subseteq)$ with $[\mathbf{B}]\notin c_\filter$.
%We will find $\mathbf{C}\in c_\filter$ such that $\vect{A}\precsim \vect{B}+\vect{C}$.
%Since $[\mathbf{B}]\notin c_\filter$, 
Then there is $n_0$ such that $B_{n_0}\in\filter$.
For each $n$, we have
\[
A_n = (A_n\cap B_{n_0}) \cup (A_n\cap B_{n_0}^c)
\subseteq B_{n_0} \cup (B_n\cap A_{n_0}^c).
\]
Set $\vect{C}:=(A_n\cap B_{n_0}^c)_n$.
Then $\vect{A}\precsim \vect{B}+\vect{C}$.
Since $(A_n\cap B_{n_0}^c)\notin\filter$, we have $\vect{C}\in c_\filter$.

By \autoref{cor:cu}, we have $\CatCuProd_\filter\{0,\infty\} \cong \CatCuProd_j\{0,\infty\} / c_\filter$.
It follows from the above argument that $\CatCuProd_\filter\{0,\infty\}$ has only two elements, $\{0,\infty\}$, as desired.
\end{proof}

%==========================================================================================
%==========================================================================================
\section{Applications}
\label{sec:appx}

%==========================================================================================
\subsection{Compact elements and K-Theory}
Recall that, if $S$ is a $\CatCu$-semigroup, then an element $x\in S$ is \emph{compact} provided $x\ll x$.
If $S$ is only a $\CatQ$-semigroup, we shall also use this terminology for an element $x\in S$ such that $x\prec x$.
The subsemigroup of compact elements of $S$ will be denoted by $S_c$.
This semigroup is a positively ordered monoid, regardless of whether $S$ is a $\CatCu$-semigroup or a $\CatQ$-semigroup.
We next describe the compact elements of an ultraproduct of $\CatCu$-semigroups, for which we use the picture of categorical ultraproducts;
see \autoref{pgr:ultraprCat}.
This will be used later to describe the Murray-von Neumann semigroup of an ultraproduct of unital \ca{s}, thus recovering \cite[Theorem~4.1]{Li05Ultraproducts} (see also \cite[Proposition~2.1]{GeHad01Ultraproducts}). 

%==========================================================================================
\begin{lma}
\label{lma:cpctelts}
Let $S$ be a $\CatQ$-semigroup.
Then, the endpoint map (\autoref{pgr:CuInQ}) induces a $\CatPom$-isomorphism $\tau(S)_c\cong S_c$.
\end{lma}
\begin{proof}
Since the endpoint map $\varphi_S\colon\tau(S)\to S$ is a $\CatQ$-morphism, we see that its restriction to $\tau(S)_c$ is a $\CatPom$-morphism that maps compact elements to compact elements.
Given $a\in S_c$, the map $f_a\colon (-\infty,0]\to S$ given by $f_a(t)=a$ for every $t$ is clearly a path with $[f_a]\in\tau(S)_c$ and $\varphi_S([f_a])=a$.
Thus, $\varphi_S$ induces an additive, order-preserving bijection $\tau(S)_c\to S_c$.

Now, let $[f],[g]\in\tau(S)_c$ satisfy $\varphi_S([f])\leq\varphi_S([g])$.
By \cite[Lemma~3.16]{AntPerThi17arX:AbsBivariantCu} applied to $[g]\ll [g]$, there is $t_0<0$ such that $g(t)\prec g(t_0)$ for all $t<0$.
Therefore $\varphi_S([g])=\sup_{t<0}g(t)\leq g(t_0)$ and thus, for every $t<0$, we have $f(t)\prec g(t_0)$.
This implies $[f]\leq [g]$ and therefore $\varphi_S$ induces an order-embedding.
\end{proof}

%==========================================================================================
\begin{cor}
\label{cor:cpcteltsinprods} Let $(S_j)_{j\in J}$ be a family of $\CatCu$-semigroups.
Then
\[
(\CatCuProd\limits_{j\in J}S_j)_c\cong\CatPomProd\limits_{j\in J}(S_j)_c\,.
\]
\end{cor}
\begin{proof}
We have $\CatCuProd_j S_j=\tau(\CatQProd_j S_j)$, by \autoref{cor:Cuproductinverselimitspullbacks}.
Therefore, \autoref{lma:cpctelts} implies that
$(\CatCuProd_j S_j)_c\cong (\CatQProd_j S_j)_c$.
Since the auxiliary relation in $\CatQProd_J S_j$ is $\llpw$, it follows that $(\CatQProd_j S_j)_c$ is isomorphic to $\CatPomProd_j (S_j)_c$.
\end{proof}

%==========================================================================================
\begin{prp}
\label{prp:ultracompact}
Let $\filter$ be an ultrafilter on a set $J$, and let $(S_j)_{j\in J}$ be a family of \CuSgp{s}.
Then there is a natural $\CatPom$-isomorphism
\[
(\prod_\filter S_j)_c\cong \prod_\filter (S_j)_c.
\]
In particular, for every compact $x\in\prod_\filter S_j$ there exists a compact $y\in\prod_j S_j$ with $x=\pi_\filter(y)$, where $\pi_\filter\colon \prod_{j\in J}S_j\to\prod_\filter S_j$ is the natural quotient map.
\end{prp}
\begin{proof}
We have $\prod_\filter S_j=\varinjlim_{F\in\filter}\prod_{j\in F}S_j$, and thus $(\prod_\filter S_j)_c\cong \varinjlim_{F\in\filter}(\prod_{j\in F}S_j)_c$ with the restriction of the same connecting maps.
Now, by \autoref{cor:cpcteltsinprods}, for each $F\in\filter$ the end point map induces an isomorphism $(\prod_{j\in F}S_j)_c\cong \prod_{j\in F}(S_j)_c$, which is easily seen to be compatible with the connecting maps.
Therefore
\[
(\prod_\filter S_j)_c
\cong \varinjlim_{F\in\filter}\big( \prod_{j\in F}S_j \big)_c
\cong \varinjlim_{F\in\filter}\prod_{j\in F}(S_j)_c
= \prod_\filter (S_j)_c.
\qedhere
\]
\end{proof}

%==========================================================================================
For a \ca{} $A$, we denote as customary by $V(A)$ the semigroup of Murray-von Neumann equivalence classes of projections in $A\otimes\KK$.
By \cite{BroCiu09IsoHilbModSF}, if $A$ is stably finite, then $V(A)$ can be identified with the semigroup of compact elements in $\Cu(A)$.
The following result for ultraproducts is implicit in \cite[Theorem~4.1]{Li05Ultraproducts}.

%==========================================================================================
\begin{cor}
\label{cor:Mvnultraproducts}
Let $\filter$ be an ultrafilter on a set $J$, and let $(A_j)_{j\in J}$ be a family of stably finite, unital \ca{s}.
Put $u_j=[1_{A_j}]\in V(A_j)$ and $u=(u_j)_{j\in J}$. 
Then:
\begin{enumerate}
\item
$V(\prod_{j\in J}A_j) \cong \big\{ x\in\prod_{j\in J}V(A_j)\colon x\leq n u\text{ for some }n\in\NN \big\}$.
\item
$V(\prod_\filter A_j) \cong \big\{ x\in\prod_\filter V(A_j)\colon x\leq n\pi_\filter(u)\text{ for some }n\in\NN \big\}$. 
\end{enumerate}
\end{cor}
\begin{proof}
First observe that both $\prod_i A_i$ and $\prod_\filter A_i$ are stably finite \ca{s}.
Therefore $V(\prod_jA_j)$ is identified with $\Cu(\prod_j A_j)_c$. 
% Let $u=[(u_i)_i]$, where $u_i=[1_{A_i}]\in\Cu(A_i)$ for each $i\in I$. 
By \autoref{thm:CuProd}, we have 
\[
\Cu(\prod_j A_j) \cong \big\{ x\in\CatCuProd_j\Cu(A_j) : x\leq \infty\, u \big\},
\]
and thus the desired conclusions follow from \autoref{cor:cpcteltsinprods}.
Statement~(2) follows analogously from \autoref{thm:CuUltraprod} and \autoref{prp:ultracompact}.
\end{proof}

%==========================================================================================
\subsection{Simplicity}

%==========================================================================================
\begin{pgr}
Recall that an ultrafilter $\filter$ is called \emph{countably complete} if it is closed under taking intersections of countably many elements.
If $\filter$ it not countably complete, then it is called \emph{countably incomplete}.
Every free ultrafilter on a countable set is countably incomplete.

It is well-known that the ultrapower of a separable \ca{} $A$ over a countably complete (free) ultrafilter is canonically isomorphic to $A$ itself;
see for example \cite[Proposition~6.1]{GeHad01Ultraproducts}.
Thus, with respect to a separable \ca{}, a countably complete, free ultrafilter behaves like a principal ultrafilter.
Let now $A$ be a simple, stable, separable, infinite-dimensional \ca{}, and let $\filter$ be a free ultrafilter.
If $\filter$ is countably incomplete, then $\prod_{\filter}A$ is neither stable (see \autoref{rmk:stableUltrapr}), nor separable, and usually not simple (see \autoref{prp:charSimpleUltraProdCa} for a characterization of simplicity of ultrapowers).
It follows that the behaviour of ultraproducts and ultrapowers is immensely different for countably complete and countably incomplete ultrafilters.

We remark that it is not obvious that countably complete, free utrafilters exist.
In fact, the existence of a countably complete, free ultrafilter is equivalent to the existence of a measurable cardinal, which puts this question outside of our usual axiomatic setting of ZFC and into the realm of set theory.
\end{pgr}

%==========================================================================================
\begin{rmk}
\label{rmk:stableUltrapr}
In general, stability does not pass to ultraproducts or ultrapowers.
More precisely, let $A$ be a (stable) \ca{}, let $I$ be a set and let $\filter$ be a free ultrafilter on $I$. A result of Kania, \cite[Theorem~1.1]{Kan15CAlgNotTensProd}, shows that a \ca{} that is a Grothendieck space is not isomorphic to the tensor product of two infinite-dimensional \ca{s}.
Moreover, he showed that $\prod_{\filter}A$ is a Grothendieck space if $\filter$ is countably incomplete.
Thus, if $\filter$ is countably incomplete and $A$ is infinite-dimensional, then $\prod_{\filter}A$ is not stable. 
If $I$ is countable, this also follows from results of Farah and Hart, \cite{FarHar13CtblSatCorona}, and Ghasemi \cite{Gha15SAWNonFactorizable}.
\end{rmk}

%==========================================================================================
\begin{dfn}
We say that a scaled \CuSgp{} $(S,\Sigma)$ has property $(C_n)$ if for all $x,y\in \Sigma$ with $y\neq 0$, we have $x\leq n y$.
\end{dfn}

%==========================================================================================
\begin{lma}
\label{prp:UltraProdCn}
Let $\filter$ be a free ultrafilter on a set $J$, and let $(S_j,\Sigma_j)_{j\in J}$ be a family of scaled \CuSgp{s}.
Let $n\in\NN$ and $F\in\filter$ such that $(S_j,\Sigma_j)$ has $(C_n)$ for all $j\in F$.
Then $\prod_{\filter} (S_j,\Sigma_j)$ has $(C_n)$.
\end{lma}
\begin{proof}
%Pick $n\geq 1$ and $F\in\filter$ such that $(S_j,\Sigma_j)$ has $(C_n)$ for every $j\in F$. 
Let $x,y\in\prod_{\filter} \Sigma_j$ with $y\neq 0$.
We need to show that $x\leq n y$.
Choose $\tilde{x},\tilde{y}\in\prod_j (S_j,\Sigma_j)$ with $\pi_\filter(\tilde{x})=x$ and $\pi_\filter(\tilde{y})=y$.
Then, choose $x_{t,j}, y_{t,j}\in S_j$ such that $\tilde{x}=[((x_{t,j})_j)_t]$ and $\tilde{y}_t=[((y_{t,j})_j)_t]$. 
%Choose $[(\vect{x}_t)_{t\leq 0}],[(\vect{y}_t)_{t\leq 0}]\in \prod_{j\in J} S_j$ representing $x$ and $y$, respectively.
Since $y\neq 0$, we have $\tilde{y}\notin c_\filter$ (see \autoref{cor:cu}).
Hence, we can fix $s<0$ such that $G:=\supp((y_{s,j})_j)$ belongs to~$\filter$.

%Note that for each $j\in F\cap G$ and $t<0$, we have $x_{t,j}, y_{s,j}\in \Sigma_j$ and $y_{s,j}\neq 0$, and therefore $x_{t,j}\leq n y_{s,j}$.

For $t\leq 0$ and $j\in J$ set
\[
z_{t,j}
:=\begin{cases}
x_{t,j}, &\text{if }j\notin F\cap G \\
0, &\text{if }j\in F\cap G
\end{cases}.
\]
Then $z:=[((z_{t,j})_j)_t]$ belongs to $c_\filter \cap \CatCuProd_j (S_j,\Sigma_j)$.
We claim that $\tilde{x} \leq n \tilde{y} + z$.
To prove this, let $t<0$.
Choose $\tilde{s}$ with $t,s<\tilde{s}<0$ (with $s$ as chosen above).

Case~1:
Let $j\in F\cap G$.
Then $S_j$ has $(C_n)$, and we have $x_{\tilde{s},j}, y_{s,j}\in \Sigma_j$ with $y_{s,j}\neq 0$.
We obtain that $x_{\tilde{s},j}\leq n y_{s,j}$, and thus
\[
x_{t,j} \ll x_{\tilde{s},j} \leq  n y_{s,j} = n y_{\tilde{s},j} + u_{\tilde{s},j}.
\]

Case~2:
Let $j\notin F\cap G$.
Then
\[
x_{t,j} \ll x_{\tilde{s},j} = z_{\tilde{s},j} \leq n y_{\tilde{s},j}+z_{\tilde s,j}.
\]

Thus, $(x_{t,j})_j \llpw n (y_{\tilde{s},j})_j + (z_{\tilde{s},j})_j$.
Hence, $\tilde{x} \leq n \tilde{y} + z$ in $\prod_j S_j$, and thus $x\leq ny$ in $\prod_{\filter} (S_j,\Sigma_j)$.
\end{proof}

%==========================================================================================
Recall that a $\CatCu$-semigroup $S$ is said to satisfy \axiomO{6} if for all $x',x,y,z\in S$ satisfying $x'\ll x\leq y+z$,
there exist $e,f\in S$ with $e\leq x,y$ and $f\leq x,z$, and such that $x'\leq e+f$. 
By \cite[Proposition~5.1.1]{Rob13Cone}, Cuntz semigroups of \ca{s} satisfy \axiomO{6}.

%==========================================================================================
\begin{lma}
\label{prp:CnSimple}
Let $(S,\Sigma)$ be a scaled \CuSgp{} satisfying \axiomO{6}. 
Assume that $(S,\Sigma)$ has $(C_n)$ for some $n\geq 1$.
Then $S$ is simple.
\end{lma}
\begin{proof}
Pick $n\geq 1$ such that $(S,\Sigma)$ has $(C_n)$.
Let $x,y\in S$ with $y\neq 0$.
We need to show that $x\leq\infty y$.

\emph{Claim: There exists $z\in\Sigma$ with $0\neq z\leq y$.}
To show this, choose $y'',y'\in S$ with $0\neq y'' \ll y' \ll y$.
Using that $\Sigma$ generates $S$ as an ideal, we can then choose $y_1,\ldots,y_K\in\Sigma$ such that
\[
y''\ll y' \leq y_1 + \ldots + y_K.
\]
Choose $y_1',\dots,y_K'$ such that $y''\ll y_K'\ll\dots\ll y_1'\ll y'$.
Applying \axiomO{6} for $y_1'\ll y'\leq y_1+(y_2+\ldots+y_K)$, we obtain $e_1$ such that $e_1\leq y_1,y'$ (in particular, $e_1\in\Sigma$) and
\[
y_1' \leq e_1 + y_2 + \ldots + y_K.
\]
In the next step, we apply \axiomO{6} for $y_2'\ll y_1'\leq y_2+(e_1+y_3+\ldots+y_K)$ to obtain $e_2\in\Sigma$ such that 
\[
y_2' \leq e_1 + e_2 + y_3 + \ldots + y_K.
\]
Iteratively, we obtain $e_1,\ldots,e_K\in\Sigma$ such that
\[
y_K' \leq e_1 + \ldots + e_K.
\]
Since $0\neq y''\leq y_K'$, there exists $k\in\{1,\ldots,K\}$ with $e_k\neq 0$.
Then $z:=e_k$ has the desired properties to verify the claim.

It is enough to verify that $x'\leq\infty y$ for every $x'\ll x$.
Thus, let $x'\ll x$.
Using again that $\Sigma$ generates $S$ as an ideal, we obtain $x_1,\ldots,x_L\in\Sigma$ such that
\[
x' \leq x_1 + \ldots + x_L.
\]
Since $(S,\Sigma)$ has $(C_n)$, we have $x_l \leq n z$ for every $l=1,\ldots,L$.
Hence, 
\[
x' 
\leq x_1 + \ldots + x_L 
\leq Ln z
\leq Ln y
\leq \infty y.\qedhere
\]
\end{proof}

%==========================================================================================
The next examples shows that \autoref{prp:CnSimple} does not hold without assuming \axiomO{6}. 

%==========================================================================================
\begin{exa}
Set $S=\{0,x,y,\infty\}$ with addition given by
\[
%0+x=x,\quad 0+y=y, \quad 0+\infty=\infty,\quad
y+y=y, \quad 
x+x=x+y=x+\infty=y+\infty=\infty+\infty=\infty,
\]
and equipped with the algebraic order (for two elements $s,t$ we set $s\leq t$ if there exists $u$ with $s+u=t$).
Note that $x$ and $y$ are incomparable.
Then $S$ is a \CuSgp.
Further, $\Sigma:=\{0,x\}$ is a scale such that $(S,\Sigma)$ satisfies $(C_1)$.
We have $\infty\nleq\infty y$, which shows that $S$ is not simple.
Note that $S$ does not satisfy \axiomO{6}:
We have $y\ll y\leq x+x$, but no elements $e,f$ satisfying $y\leq e+f$ and $e,f\leq x,y$.
\end{exa}

%==========================================================================================
\begin{lma}
\label{prp:SimpleUltraProdGivesCn}
Let $\filter$ be a countably incomplete, free ultrafilter on a set $J$, and let $(S_j,\Sigma_j)_{j\in J}$ be a family of scaled \CuSgp{s} such that $\prod_{\filter} (S_j,\Sigma_j)$ is simple.
Then there exist $n\in\NN$ and $F\in\filter$ such that $S_j$ has $(C_n)$ for all $j\in F$.
\end{lma}
\begin{proof}
We argue by contradiction.
Thus, assume that for all $n\in\NN$, and for all $F\in\filter$, there is $j\in F$ such that $S_j$ does not have $(C_n)$.
Let us write $c_\filter$ for $c_\filter((S_j)_j)$.
By \autoref{cor:cu}, we have $\prod_{\filter} S_j \cong (\prod_j S_j)/c_\filter$, and we let $\pi_\filter\colon\prod_j S_j\to \prod_{\filter} S_j$ denote the quotient map.

Since $\filter$ is countably incomplete, we may find a decreasing sequence $(F_n)_n$ of elements in $\filter$ such that $\bigcap_n F_n=\emptyset$.
By our assumption the sets 
\[
G_ n:= \big\{ j\in F_n : S_j \text{ does not have } (C_n) \big\}
\]
also belong to $\filter$, and they also form a decreasing sequence with empty intersection.

Replacing $F_n$ by $G_n$ we may assume that for every $j\in F_n$, the semigroup $S_j$ does not have $(C_n)$.
For each $n\geq 1$, and each $j\in F_n\setminus F_{n+1}$, choose elements $x_j$, $y_j\in \Sigma_j$ with $y_j\neq 0$ and $x_j\not\leq n y_j$.

Further, choose paths $(x_{t,j})_{t\leq 0}$ in $S_j$ with $x_j=x_{0,j}$ such that $x_{t,j}\not\leq ny_j$ for all $t\leq 0$.
Likewise, choose paths $(y_{t,j})_{t\leq 0}$ in $S_j$ with $y_j=y_{0,j}$ and such that $y_{t,j}\neq 0$ for all $t\leq 0$.
For $t\leq 0$ set
\[
\vect{x}_t
:=\begin{cases}
x_{t,j}, &\text{if }j\in F_n\setminus F_{n+1} \\
0, &\text{otherwise};
\end{cases}
\andSep
\vect{y}_t
:=\begin{cases}
y_{t,j}, &\text{if }j\in F_n\setminus F_{n+1} \\
0, &\text{otherwise.}
\end{cases}
\]
Set $x=[(\vect{x}_t)_{t\leq 0}]$ and $y=[(\vect{y}_t)_{t\leq 0}]$ in $\prod_j S_j$.
By construction, we have $x,y\in \prod_j(S_j,\Sigma_j)$.
Since $\supp(\vect{y}_{-1})=F_1\in\filter$, we have $y\notin c_\filter$, and therefore $\pi_\filter(y)\neq 0$.
Using that $\prod_\filter (S_j,\Sigma_j)$ is simple, it follows that $\pi_\filter(x)\leq\infty\pi_\filter(y)$.

Let $\varepsilon>0$, and let $x_\varepsilon:=[(\vect{x}_{t-\varepsilon})_{t\leq 0}]$ be the $\varepsilon$-cut-down of $x$.
We have $x_\varepsilon \ll x$ in $\prod_j(S_j,\Sigma_j)$, and therefore
\[
\pi_\filter (x_\varepsilon)
\ll \pi_\filter(x)
\leq \infty \pi_\filter(y)
= \sup_m m\pi_\filter(y),
\]
in $\prod_\filter(S_j,\Sigma_j)$.
We obtain $m\in\NN$ such that $\pi_\filter(x_\varepsilon) \leq m\pi_\filter(y)$, and thus there is $z\in c_\filter$ such that $x_\varepsilon \leq m y +z$.
%$[(\vect{z}_t)_{t\leq 0}]\in c_\filter$ such that $[[(\vect{x}_{t-\varepsilon})_{t\leq 0}]\leq m [(\vect{y}_t)_{t\leq 0}]+[(\vect{z}_t)_{t\leq 0}]$.
Choose $z_{t,j}\in S_j$ such that $z=[((z_{t,j})_{j\in J})_{t\leq 0}]$.

Fix any $t<0$.
Then there is $s<0$ such that $\vect{x}_{t-\varepsilon}\llpw m\vect{y}_s+\vect{z}_s$. 
%In particular, we find $s<0$ such that $\vect{x}_{-2\varepsilon}\llpw m\vect{y}_s+\vect{z}_s$.
Set $T:=(\supp (\vect{z}_s))^c$.
Since $z\in c_\filter$, we have $T\in\filter$, and thus there is $N\geq m$ with $T\cap (F_N\setminus F_{N+1})\neq\emptyset$.
Choose $j$ in $T\cap (F_N\setminus F_{N+1})$.
Then $z_{s,j}=0$, and thus
\[
(\vect{x}_{t-\varepsilon})_j=x_{t-\varepsilon,j}\leq my_{s,j}\leq my_j\leq N y_j,
\]
in contradiction with our choice.
\end{proof}

%==========================================================================================
\begin{thm}
\label{prp:charSimpleGenUltraProd}
Let $\filter$ be a countably incomplete, free ultrafilter on a set $J$, and let $(S_j,\Sigma_j)_{j\in J}$ be a family of scaled \CuSgp{s} that satisfy \axiomO{6}.
Then the following are equivalent:
\begin{enumerate}
\item 
$\prod_{\filter} (S_j,\Sigma_j)$ is simple.
\item 
$\prod_{\filter} (S_j,\Sigma_j)$ has $(C_n)$ for some $n\geq 1$.
\item 
There exist $n\in\NN$ and $F\in\filter$ such that $S_j$ has $(C_n)$ for all $j\in F$.
\end{enumerate}

In particular, if $(S,\Sigma)$ is a scaled \CuSgp{} satisfying \axiomO{6}, then the ultrapower $\prod_{\filter}(S,\Sigma)$ is simple if and only if $(S,\Sigma)$ has $(C_n)$ for some $n$.
\end{thm}
\begin{proof}
By \autoref{prp:SimpleUltraProdGivesCn}, (1) implies~(3), and by \autoref{prp:UltraProdCn} (3) implies~(2).
It is straightforward to verify that \axiomO{6} passes to products and quotients in $\CatCu$.
Hence, $\prod_{\filter} (S_j,\Sigma_j)$ satisfies \axiomO{6}, whence we can apply \autoref{prp:CnSimple} to deduce that~(2) implies~(1).
\end{proof}

%==========================================================================================
Let $n\in\NN$.
Adapting \cite[Definition~4.3]{KirRor02InfNonSimpleCalgAbsOInfty} from \ca{s} to \CuSgp{s}, let us say that a \CuSgp{} $S$ is \emph{pi-$n$} if $2nx=nx$ for every $x\in S$.
Further, we say that $S$ is \emph{weakly purely infinite} if it is pi-$n$ for some $n$.

Note that a \CuSgp{} $S$ is simple and pi-$n$ if and only if for all $x,y\in S$ with $y\neq 0$ we have $x\leq ny$.
Thus, $S$ is simple and pi-$n$ if and only if the trivially scaled \CuSgp{} $(S,S)$ has $(C_n)$.

%==========================================================================================
\begin{cor}
Let $\filter$ be a countably incomplete, free ultrafilter on a set $J$, and let $(S_j)_{j\in J}$ be a family of \CuSgp{s}.
Then the following are equivalent:
\begin{enumerate}
\item 
$\prod_{\filter} S_j$ is simple.
\item 
$\prod_{\filter} S_j$ is simple and weakly purely infinite.
\item 
There exist $n\in\NN$ and $F\in\filter$ such that $S_j$ is simple and pi-$n$ for all $j\in F$.
\end{enumerate}

In particular, the ultrapower $\prod_{\filter}S$ of a \CuSgp{} $S$ is simple if and only if~$S$ is simple and weakly purely infinite.
\end{cor}
\begin{proof}
It is clear that~(2) implies~(1).
As in the proof of \autoref{prp:charSimpleGenUltraProd}, we obtain that~(1) implies~(3), and that~(3) implies~(2).
(Note that Lemmas~\ref{prp:SimpleUltraProdGivesCn} and~\ref{prp:UltraProdCn} do not require \axiomO{6}.)
\end{proof}

%==========================================================================================
\begin{lma}
\label{prp:charCaCn}
Let $A$ be a \ca{}, and let $n\geq 1$.
Then $(\Cu(A),\Sigma_A)$ satisfies $(C_n)$ if and only if $A\cong M_k(\CC)$ for $k\leq n$, or $A$ is simple and purely infinite.
\end{lma}
\begin{proof}
The backward implication is easy to show.
Thus, assume that $(\Cu(A),\Sigma_A)$ satisfies $(C_n)$.
Since Cuntz semigroups of \ca{s} satisfy \axiomO{6}, it follows from \autoref{prp:CnSimple} that $A$ is simple.
Then $A$ is either elementary (that is, $A$ contains a minimal nonzero projection), or not.
In the first case, we obtain $\Cu(A)\cong\NNbar$.
It is easy to verify that $(\NNbar,\Sigma)$ has $(C_n)$ if and only if $\Sigma\subseteq\{0,\ldots,n\}$.
It follows that $A$ contains only projections of rank at most $n$, and thus $A\cong M_k(\CC)$ for some $k\leq n$.

In the other case, it follows from the Glimm Halving Lemma that for every nonzero $x\in\Cu(A)$ there exists a nonzero $y$ with $2y\leq x$.
Let $x\in\Cu(A)$ be nonzero.
From the Claim in the proof of \autoref{prp:CnSimple} we obtain a nonzero $y\in\Sigma_A$ with $y\leq x$.
Applying the Glimm Halving Lemma several times, we obtain a nonzero $z$ with $2nz\leq y$.
By $(C_n)$, we also have $y\leq nz$, and thus $2y\leq 2nz\leq y$.
Therefore $y=\infty$, and then $x=\infty$.
It follows that $\Cu(A)\cong\{0,\infty\}$, which implies that $A$ is simple and purely infinite.
\end{proof}

%==========================================================================================
Using \autoref{prp:charSimpleGenUltraProd}, we obtain a characterization for when an ultraproduct of \ca{s} is simple.
In the case that we have a purely infinite simple \ca{} $A$, and a free ultrafilter $\filter$ over $\NN$, it was proved by Kirchberg (see \cite[Proposition~6.2.6]{Ror02Classification}) that the ultrapower $A_\filter$ is again purely infinite simple.
The result that an ultraproduct of purely infinite, simple \ca{s} is again purely infinite simple can be obtained as a consequence of \cite[Theorem~2.5.1]{FarHarLupRobTikVigWin06arX:ModelThy}.

%==========================================================================================
\begin{thm}
\label{prp:charSimpleUltraProdCa}
Let $\filter$ be a countably incomplete, free ultrafilter on a set $J$, and let $(A_j)_{j\in J}$ be a family of \ca{s}.
Then $\prod_{\filter}A_j$ is simple if and only if either $A_j$ is simple and purely infinite for almost all $j$ in $J$, or else there is $n\in\NN$ such that $A_j\cong M_n(\CC)$ for almost all $j$ in $J$.

In particular, the ultrapower $\prod_{\filter}A$ of a \ca{} $A$ is simple if and only if $A$ is either simple and purely infinite or $A\cong M_n(\CC)$ for some $n\in\NN$.
\end{thm}
\begin{proof}
If $A_j\cong M_n(\CC)$ for almost all $j\in J$, then $\prod_{\filter}A_j \cong M_n(\CC)$, which is simple. 
Note that a \ca{} $A$ is simple and purely infinite if and only if $\Cu(A)\cong\{0,\infty\}$, with the trivial scale.
Thus, if $A_j$ is simple and purely infinite for almost all $j\in J$, then using \autoref{thm:CuProd} and \autoref{prp:prodE0} we obtain
\[
\Cu(\prod_\filter A_j)
\cong \prod_\filter \Cu(A_j)
\cong \prod_\filter \{0,\infty\}
\cong\{0,\infty\},
\]
which implies that $\prod_\filter A_j$ is simple (and purely infinite).

To show the converse implication, assume that $\prod_\filter A_j$ is simple.
Since Cuntz semigroups of \ca{s} satisfy \axiomO{6}, and since the Cuntz semigroup of a \ca{} is simple if and only if the \ca{} is simple, we can apply \autoref{prp:charSimpleGenUltraProd} to deduce that there exist $n\geq 1$ and $F\in\filter$ such that the scaled \CuSgp{s} $\Cu(A_j)$ have $(C_n)$ for all $j\in F$.
Let $F_0$ be set of $j\in F$ such that $A_j$ is simple and purely infinite, and for $k=1,\ldots,n$ let $F_k$ be the set of $j\in F$ such that $A_j\cong M_k(\CC)$.
By \autoref{prp:charCaCn}, we have $F=F_0\cup\ldots\cup F_n$, and since $\filter$ is an ultrafilter, we have $F_k\in\filter$ for precisely one $k$, which proves the claim.
\end{proof}

%==========================================================================================
It is worth mentioning that a similar characterization is obtained by Kirchberg \cite{Kir06CentralSeqPI} for central sequence \ca{s}, proving that $(A_{\mathcal U}\cap A')/\text{Ann}_A(A_{\mathcal U})$ is simple if and only if $A\otimes \mathcal K\cong \mathcal K$ or $A$ is simple purely infinite and nuclear.

%==========================================================================================
\subsection{Comparability}

%==========================================================================================
We show that several comparability properties pass to (ultra)products of \CuSgp{s}.
In particular, in \autoref{prp:ssaTotalOrder} we recover an unpublished result of A. Tikuisis, that was brought to our attention by L. Robert (private communication).
We refer to \cite[Section~5.6, p.72ff]{AntPerThi18TensorProdCu}, for definitions of the unperforation properties considered in the next results.

%==========================================================================================
\begin{prp}
\label{thm:comparability}
The following properties pass to products, direct sums, and ultraproducts of \CuSgp{s}:
unperforation, almost (near) unperforation.
\end{prp}
\begin{proof}
It is straightforward to verify that each of the considered properties passes to ideals and quotients.
Thus, it is enough to consider products.
We only prove the statement for unperforation, as the other situations are similar.

Let $(S_j)_{j\in J}$ be a family of unperforated \CuSgp{s}, and let $x,y\in\prod_jS_j$ and $n\geq 1$ satisfy $nx\leq ny$.
Choose $\vect{x}_t=(x_{t,j})_j$ and $\vect{y}_t=(y_{t,j})_j$ such that $x=[(\vect{x}_t)_{t\leq 0}]$ and $y=[(\vect{y}_t)_{t\leq 0}]$.
Given $t<0$, there is $s<0$ such that $n \vect{x}_{t}\llpw n \vect{y}_s$, that is, $nx_{t,j}\ll ny_{s,j}$ for all $j\in J$.
Therefore $x_{t,j}\ll y_{s,j}$ for all $j$, which implies $\vect{x}_t \llpw \vect{y}_s$.
Hence, $x\leq y$, as desired.
\end{proof}

%==========================================================================================
\begin{cor}
Given a family of \ca{s} whose Cuntz semigroups are unperforated (almost unperforated, nearely unperforated), the Cuntz semigroups of their product, direct sum, and ultraproduct over some ultrafilter are also unperforated (almost unperforated, nearly unperforated).
\end{cor}
\begin{proof}
This is a direct consequence of \autoref{thm:comparability}, using Theorems~\ref{thm:CuProd}, \ref{thm:coproduct} or \ref{thm:CuUltraprod}, and using that each of the properties passes to ideals, and hence from (ultra)products to scaled (ultra)products.
\end{proof}

%==========================================================================================
\begin{prp}
\label{prp:totalOrder}
Let $\filter$ be an ultrafilter on a set $J$, and let $(S_j)_{j\in J}$ be a family of totally ordered \CuSgp{s}.
Then $\CatCuProd_\filter S_j$ is totally ordered.
\end{prp}
\begin{proof}
Let $x,y\in\prod_jS_j$.
We need to show that $\pi_\filter(x)\leq \pi_\filter(y)$ or $\pi_\filter(y)\leq \pi_\filter(x)$.
Choose $\vect{x}_t=(x_{t,j})_j$ and $\vect{y}_t=(y_{t,j})_j$ such that $x=[(\vect{x}_t)_{t\leq 0}]$ and $y=[(\vect{y}_t)_{t\leq 0}]$.
For every $\varepsilon>0$ and $t<0$ let us define the following set:
\[
F_{\varepsilon,t} := \big\{ j\in J : x_{-\varepsilon,j}\leq y_{t,j} \big\}.
\]

If $F_{\varepsilon,t}\in\filter$ for some $\varepsilon>0$ and $t<0$, then $\pi_\filter([(\vect{x}_{t-\varepsilon})_{t\leq 0}]) \leq \pi_\filter([(\vect{y}_t)_{t\leq 0})$.
Hence, if for all $\varepsilon>0$ there is $t<0$ with $F_{\varepsilon, t}\in\filter$, then $\pi_\filter(x)\leq \pi_\filter(y)$.
Otherwise, there is $\varepsilon>0$ with $F_{\varepsilon,t}\notin\filter$ for all $t<0$.
Since the $S_j$ are totally ordered, we get
\[
F_{\varepsilon,t}^c = \big\{ j\in J : y_{t,j}\leq x_{-\varepsilon, j}\}\in \mathcal U,
\]
for all $t<0$, which then implies $\pi_\filter(y)\leq \pi_\filter(x)$.
\end{proof}

%==========================================================================================
\begin{cor}
\label{prp:ssaTotalOrder}
Let $\filter$ be an ultrafilter over some set, and let $A$ be a strongly self-absorbing, stably finite \ca{} satisfying the UCT.
Then the Cuntz semigroup of the ultrapower $A_\filter$ is totally ordered.
\end{cor}
\begin{proof}
By \cite[Paragraph~7.6.1, p.136]{AntPerThi18TensorProdCu}, the Cuntz semigroup $\Cu(A)$ is totally ordered.
By \autoref{prp:totalOrder}, $\prod_\filter\Cu(A)$ is totally ordered.
Now, the statement follows using that $\Cu(\prod_\filter A)$ is an ideal in $\prod_\filter\Cu(A)$;
see \autoref{thm:CuUltraprod}.
\end{proof}

%==========================================================================================
\begin{rmk}
Recall that a functional on a $\CatCu$-semigroup $S$ is a $\CatPom$-morphism $S\to[0,\infty]$ that preserves suprema of increasing sequences.
Given a family $(S_j)_{j\in J}$ of $\CatCu$-semigroups and an ultrafilter $\filter$ on $J$, it is natural to study functionals on the ultraproduct $\prod_{\filter}S_j$.
Given a functional $\varphi_j\colon S_j\to[0,\infty]$, for each $j\in J$, there is a naturally induced functional on $\prod_{\filter}S_j$.
The construction of this functional and the question whether the set of such functionals is dense in the set of all functionals on $\prod_{\filter}S_j$ will be analysed in detail in forthcoming work, \cite{AntPerRobThi20pre:TracesUltra}.
\end{rmk}

%==========================================================================================
%\bibliographystyle{../../aomalphaMyShort}
%\bibliography{../../References}

\providecommand{\etalchar}[1]{$^{#1}$}
\providecommand{\bysame}{\leavevmode\hbox to3em{\hrulefill}\thinspace}
\providecommand{\noopsort}[1]{}
\providecommand{\mr}[1]{\href{http://www.ams.org/mathscinet-getitem?mr=#1}{MR~#1}}
\providecommand{\zbl}[1]{\href{http://www.zentralblatt-math.org/zmath/en/search/?q=an:#1}{Zbl~#1}}
\providecommand{\jfm}[1]{\href{http://www.emis.de/cgi-bin/JFM-item?#1}{JFM~#1}}
\providecommand{\arxiv}[1]{\href{http://www.arxiv.org/abs/#1}{arXiv~#1}}
\providecommand{\doi}[1]{\url{http://dx.doi.org/#1}}
\providecommand{\MR}{\relax\ifhmode\unskip\space\fi MR }
% \MRhref is called by the amsart/book/proc definition of \MR.
\providecommand{\MRhref}[2]{%
  \href{http://www.ams.org/mathscinet-getitem?mr=#1}{#2}
}
\providecommand{\href}[2]{#2}

\end{document}